\documentclass{amsart}
\usepackage{latexsym,amsxtra,amscd,ifthen,amsmath,color,hyperref,float}
\usepackage{amsfonts}
\usepackage{verbatim}
\usepackage{amsmath}
\usepackage{amsthm}
\usepackage{amssymb}
\usepackage{fancybox}
\usepackage[notcite,notref,final]{showkeys}
 \usepackage[all,cmtip]{xy}
 \usepackage[normalem]{ulem}

\numberwithin{equation}{section}

\theoremstyle{plain}
\newtheorem{theorem}{Theorem}[section]
\newtheorem{lemma}[theorem]{Lemma}

\newtheorem{proposition}[theorem]{Proposition}

\newtheorem{corollary}[theorem]{Corollary}
\newtheorem{conjecture}[theorem]{Conjecture}

\theoremstyle{definition}
\newtheorem{definition}[theorem]{Definition}
\newtheorem{example}[theorem]{Example}

\newtheorem{remark}[theorem]{Remark}
\newtheorem{question}[theorem]{Question}

\newcommand{\mf}{\mathfrak}
\newcommand{\mc}{\mathcal}

\newcommand{\id}{\text{id}}

\newcommand{\kk}{{\Bbbk}}

\makeatletter              
\let\c@equation\c@theorem  
\makeatother

\begin{document}

\title[Pointed Hopf actions on fields, I]
{Pointed Hopf actions on fields, I}

\author{Pavel Etingof and Chelsea Walton}

\address{Etingof: Department of Mathematics, Massachusetts Institute of Technology, Cambridge, Massachusetts 02139,
USA}

\email{etingof@math.mit.edu}

\address{Walton: Department of Mathematics, Massachusetts Institute of Technology, Cambridge, Massachusetts 02139,
USA}

\email{notlaw@math.mit.edu}

\bibliographystyle{abbrv}       

\begin{abstract}
Actions of semisimple Hopf algebras $H$ over an algebraically
closed field of characteristic zero on commutative
domains were classified recently by the authors in
\cite{EtingofWalton:ssHopf}. The answer turns out to be very
simple --  if the action is inner faithful, then $H$ has to be a
group algebra. The present article contributes to the
non-semisimple case, which is much more complicated. 
Namely, we study actions of finite dimensional (not necessarily semisimple) 
Hopf algebras on commutative domains, particularly when $H$ is
pointed of finite Cartan type. 

The work begins by reducing to the case where $H$ acts inner 
faithfully on a field; such a Hopf algebra is referred to as {\it Galois-theoretical}. 
We present examples of such Hopf algebras, which include the Taft
algebras, $u_q(\mf{sl}_2)$, 
and some Drinfeld twists of other small quantum groups. 
We also give many examples of finite dimensional Hopf algebras 
which are not Galois-theoretical. 
Classification results on  finite dimensional pointed 
Galois-theoretical Hopf algebras of finite Cartan type will be provided in the sequel, Part II, of this study. 
\end{abstract}

\subjclass[2010]{13B05, 16T05, 81R50}
\keywords{commutative domain, field, finite Cartan type, Galois-theoretical, Hopf algebra action, pointed}

\maketitle



\section{Introduction} \label{sec:intro}

Let $\kk$ be an algebraically closed field of characteristic zero, and let an unadorned $\otimes$ denote $\otimes_{\kk}$.
This work contributes to the field of noncommutative invariant theory in the sense of studying quantum analogues of group actions on commutative $\kk$-algebras.  Here, we restrict our attention to the actions of  finite quantum groups, i.e. finite dimensional Hopf algebras, as these objects  and their actions on (quantum) $\kk$-algebras have been the subject of recent research in noncommutative invariant theory, including \cite{Artamonov2003}, \cite{CWWZ:filtered}, \cite{Etingof:PIintegrality}, \cite{EtingofWalton:ssHopf}, \cite{KKZ:Gorenstein}, \cite{LWZ}, \cite{MontgomerySchneider}, \cite{MontgomerySmith}, \cite{Skryabin:invariants}. 
The two important classes of finite dimensional Hopf algebras $H$ are those that are {\it semisimple} (as a $\Bbbk$-algebra) and those that are {\it pointed} (namely, all simple $H$-comodules are 1-dimensional).
Moreover, we have many choices of what one could consider to be a quantum $\kk$-algebra, but from the viewpoint of classical invariant theory and algebraic geometry,  the examination of Hopf actions on commutative domains over $\kk$ is of interest. Since the classification of semisimple Hopf actions on commutative domains over $\kk$ is  understood by work of the authors \cite{EtingofWalton:ssHopf}, the focus of this article is to classify finite dimensional non-semisimple Hopf ($H$-) actions on commutative domains over $\kk$, particularly when $H$ is pointed.

It was announced in the latest survey article of Nicol\'{a}s Andruskiewitsch \cite{And:ICM} that the classification of finite dimensional pointed Hopf algebras
$H$, in the case when $H$ has an abelian group $G$ of
grouplike elements, is expected to be completed soon. 
In this
case, $H$ is a lifting of the bosonization of a Nichols algebra
$\mf{B}(V)$ of diagonal type by the group algebra $\Bbbk G$, that
is to say, $gr(H) \cong \mf{B}(V) \# \Bbbk G$. 
The most
extensively studied class of Nichols algebras are those of {\it finite Cartan type};
their bosonizations
are variations of Lusztig's small quantum groups. 
Prompted by the main classification result of finite dimensional pointed Hopf algebras of finite
Cartan type, provided by Andruschiewitsch-Schneider in \cite{AndSch:pointed}, we restrict our attention to the actions of such Hopf algebras on commutative domains.

All Hopf algebra actions in this work, unless otherwise specified, are assumed {\it inner faithful} in the sense that the action does not factor through a `smaller' Hopf algebra [Definition~\ref{def:infaith}].

We begin our study of Hopf actions on commutative domains by reducing to the case where Hopf algebras act inner faithfully on fields [Lemma~\ref{lem:Skryabin}, Remark~\ref{rem:liftcomdomain}]; such Hopf algebras are referred as {\it Galois-theoretical} [Definition~\ref{def:GT}]. A general result on Galois-theoretical Hopf algebras is as follows.

\begin{proposition} [Propositions~\ref{prop:GTprelim} and~\ref{prop:GTtwist}] The Galois-theoretical property is preserved under taking a Hopf subalgebra, and preserved under tensor product, but is not preserved under Hopf dual, 2-cocycle deformation (that alters multiplication), nor Drinfeld twist (that alters comultiplication).
\end{proposition}

Examples of Galois-theoretical Hopf algebras include all finite
group algebras, and moreover, any semisimple Galois-theoretical
Hopf algebra is a group algebra
[Proposition~\ref{prop:GTprelim}(b)]. In contrast to this, we
will see below that there are many examples
of non-semisimple finite dimensional pointed Galois-theoretical Hopf algebras, particularly of finite Cartan type. 

\begin{theorem} \label{thm:main}
Let $q\neq \pm 1$ be a root of unity unless stated otherwise, and
let $\mf{g}$ be a finite dimensional simple Lie algebra. 

(1) The following are \underline{examples} of Galois-theoretical
finite dimensional pointed Hopf algebras of finite Cartan type.

\medskip

\begin{tabular}{|llclr|}
\hline
\underline{Galois-theoretical Hopf algebra} && \underline{finite Cartan type} && \underline{Reference}\\
Taft algebras $T(n)$ &&  A$_1$ && Prop.~\ref{prop:T(n)}\\
Nichols Hopf algebras $E(n)$ &&  A$_1^{\times n}$ &&Prop.~\ref{prop:E(n)}\\
the book algebra $\mathbf{h}(\zeta, 1)$&& A$_1 \times$A$_1$ && Prop.~\ref{prop:book}\\
the Hopf algebra $H_{81}$ of dimension 81 && A$_2$ && Prop.~\ref{prop:H81}\\
 $u_q(\mf{sl}_2)$ && A$_1 \times$A$_1$ && Prop.~\ref{prop:uqsl2}\\
 $u_q(\mf{gl}_2)$ && A$_1 \times$A$_1$ && Prop.~\ref{prop:uq(gl2)}\\
Twists $u_q(\mf{gl}_n)^{J^+}$, $u_q(\mf{gl}_n)^{J^-}$ for $n \geq 2$&& A$_{n-1} \times$A$_{n-1}$ && Prop.~\ref{prop:uqglntwist}\\
Twists $u_q(\mf{sl}_n)^{J^+}$, $u_q(\mf{sl}_n)^{J^-}$ for $n \geq 2$&& A$_{n-1} \times$A$_{n-1}$ && Cor.~\ref{cor:uqslntwist}\\
Twists $u_q^{\geq 0}(\mf{g})^J$ for $2^{\text{rank}(\mf{g})-1}$ of such $J$ &&  same type as $\mf{g}$ && Prop.~\ref{prop:uqgpostwist}\\
\hline
\end{tabular}
\smallskip

\noindent For the last three cases, $q$ is a root of unity of odd order $m \geq 3$, with $m>3$ for type G$_2$. Further, $n$ is relatively prime to $m$ for the result on $u_q(\mf{sl}_n)^{J^+}$ and $u_q(\mf{sl}_n)^{J^-}$.

\medskip

(2) The following are \underline{non-examples} of Galois-theoretical finite dimensional pointed Hopf algebras of finite Cartan type.
\bigskip

\hspace{-.22in}
\begin{tabular}{|lcll|}
\hline
\underline{Non-Galois-theoretical Hopf algebra} & \underline{finite Cartan type} && \underline{Reference}\\
generalized Taft algs which are not Taft algs &  A$_1$ && Prop.~\ref{prop:genTaft}\\
the book algebra $\mathbf{h}(\zeta, p)$ for $p \neq 1$ & A$_1 \times$A$_1$ && Prop.~\ref{prop:book}\\
 gr($u_q(\mf{sl}_2)$) & A$_1 \times$A$_1$ && Prop.~\ref{prop:uqsl2}\\
\hline
\end{tabular}

\bigskip
\end{theorem}
This theorem will be used in Part II of this work on the classification of Galois-theoretical Hopf algebras of finite Cartan type. See Remarks~\ref{rem:fullII} and~\ref{rem:posII} for a preview of these results for $u_q(\mf{gl}_n)$, $u_q(\mf{sl}_n)$, $u_q^{\geq 0}(\mf{g})$, and their twists. 

For each of the Galois-theoretical Hopf algebras $H$ in the theorem above, the module fields $L$ are analyzed in terms of their invariant subfields $L^H$. For instance, we have the following result.

\begin{theorem} [Theorem \ref{inva}] Let $H$ be a
  finite dimensional pointed 
Galois-theoretical Hopf algebra with
  $H$-module field $L$. 
Then, for the group  $G=G(H)$ of grouplike elements of $H$, we have the following statements:
\begin{enumerate}
\item $L^H = L^G$; and
\item the extension $L^H \subset L$ is Galois with Galois group $G$.
\end{enumerate}
\end{theorem}

A generalization of this result is provided for finite dimensional Hopf actions on commutative domains (Theorem~\ref{inva2}) and on Azumaya algebras (Theorem~\ref{inva3}).

Further, we point out that finite dimensional Galois-theoretical Hopf algebras are not necessarily pointed nor semisimple [Example~\ref{ex:nonptedGT}]. 

\begin{remark} It is interesting to consider the quasiclassical analogue of our study of finite dimensional Hopf actions on commutative domains (that are faithful in some sense). 
To do so, let $G$ be a Poisson algebraic group and let $X$ be an irreducible algebraic variety with zero Poisson bracket. Then, the corresponding problem is to (1) determine which of such $G$ can have a faithful Poisson action on a variety $X$ as above, and (2) classify such actions.
In particular, for $G'$ a closed subgroup of $G$, this includes the problem of classifying Poisson homogeneous spaces $X=G/G'$
which have zero Poisson bracket. Here, the Poisson bracket on the group
$G$ is not necessarily zero. See \cite{Drinfeld1993}, \cite{Karolinsky2001}, and \cite{Zak1995} for further reading.
\end{remark}

This paper is organized as follows. Background material on pointed Hopf algebras and Hopf algebra actions is provided in Section~\ref{sec:background}. This includes a discussion of Hopf algebras of finite Cartan type, of quantum groups at roots of unity, and of twists of Hopf algebras and Hopf module algebras. We define and provide preliminary results on the Galois-theoretical property in Section~\ref{sec:prelim}. The proof of Theorem~\ref{thm:main} is established in Section~\ref{sec:examples} via Propositions~\ref{prop:T(n)}, \ref{prop:E(n)}, \ref{prop:genTaft}--\ref{prop:H81}, \ref{prop:uqsl2}, \ref{prop:uq(gl2)}, \ref{prop:uqglntwist}, \ref{prop:uqgpostwist}, and Corollary~\ref{cor:uqslntwist}.


\section{Background material} \label{sec:background}

In this section, we provide a background discussion of pointed Hopf algebras, especially those of finite Cartan type (Section~\ref{subsec:pointed}) and quantum groups at roots of unity (Section~\ref{rootof1}). We also discuss Hopf algebra actions on $\kk$-algebras (Section~\ref{subsec:actions}) and Drinfeld twists of these actions (Section~\ref{subsec:twist}). Consider the notation below, which will be explained in the following discussion. Unless specified otherwise: 
\medskip

\hspace{.1in} $\kk$ =  an algebraically closed base field of characteristic zero\\
$\zeta$, $q$, $\omega$ = a primitive root of unity in $\kk$ of order $n$, $m$, and 3, respectively\\
\indent \hspace{.06in} $H$ = a  finite dimensional Hopf algebra with coproduct $\Delta$, counit $\varepsilon$, antipode $S$\\
\indent \hspace{.06in} $G$ = the group of grouplike elements $G(H)$ of $H$\\
\indent \hspace{.06in} $\widehat{G}$ = character group of $G$ = $\{\alpha: G \rightarrow \kk^{\times}\}$\\
\indent \hspace{.06in} $A$ = an $H$-module algebra over $\Bbbk$\\
\indent \hspace{.06in} $L$ = an $H$-module field containing  $\kk$\\
\indent \hspace{.06in} $F$ = the subfield of invariants $L^H$

%
%
%

\subsection{Grouplike and skew primitive elements, and pointed Hopf algebras}
\label{subsec:pointed}
Consider the following notation and terminology.
A nonzero element $g \in H$ is {\it grouplike} if $\Delta(g) = g \otimes g$, and the group of grouplike elements of $H$ is denoted by $G=G(H)$. An element $x \in H$ is 
{\it $(g,g')$-skew primitive}, if for grouplike elements $g, g'$ of $G(H)$, we have that $\Delta(x) = g \otimes x + x \otimes g'$. The space of such elements is denoted by $P_{g,g'}(H)$.

The {\it coradical} $H_0$ of  a Hopf algebra $H$ is the sum of all simple subcoalgebras of $H$. The {\it coradical filtration} $\{H_n\}_{n \geq 0}$ of $H$ is defined inductively by 
$$H_n = \Delta^{-1}(H \otimes H_{n-1} + H_0 \otimes H),$$
where $H= \bigcup_{n \geq 0} H_n$.

We say that a Hopf algebra $H$ is {\it pointed}  if all of its simple $H$-comodules (or equivalently, if all of its simple $H$-subcoalgebras) are 1-dimensional. When $H$ is pointed, we have that $H_0 = \kk  G$ and $H_1=\kk G + \left( \sum_{g,g'\in G}P_{g,g'}(H)\right)$. Although this sum is not direct, one has 
$H_1/H_0=\bigoplus_{g,g'\in G}\overline{P}_{g,g'}(H)$, where $\overline{P}_{g,g'}(H)$ is the image of $P_{g,g'}(H)$ in $H_1/H_0$. 
One can verify easily the following result.

\begin{lemma} \label{lem:pointed} (a) The coradical $H_0$ of a Hopf algebra $H$ is the group algebra $\kk G(H)$ if and only if $H$ is pointed.

\smallskip

\noindent (b) If a Hopf algebra $H$ is generated by grouplike and skew primitive elements, then $H$ is pointed. 
\qed
\end{lemma}

The converse of part (b) is expected in the finite dimensional case.

\begin{conjecture} \label{conj:AS} \cite[Conjecture~1.4]{AndSch:finite} 
A finite dimensional pointed Hopf algebra over an algebraically closed field of characteristic zero is generated by grouplike and skew primitive elements.
\end{conjecture}

In fact, the conjecture holds in the setting of our work.

\begin{theorem} \label{thm:ASconjAbel} \cite[Theorem~4.15]{Angiono}
Conjecture~\ref{conj:AS} holds when $G$ is abelian. \qed
\end{theorem}

As a consequence, a finite dimensional pointed Hopf algebra $H$ over $\kk$ is a {\it lifting} of a bosonization of a Nichols algebra $\mf{B}(V)$ by the group $G$.  In other words, $gr(H) \cong \mf{B}(V) \# \kk G$ in this case.
Moreover, we consider a special subclass of pointed Hopf algebras, that of {\it finite Cartan type}. Refer to \cite{AndSch:survey} and \cite{AndSch:pointed} for further details.

\begin{definition}{} \label{def:V,c} 
Let $(V,c)$ be a finite dimensional braided vector space. 
\begin{itemize}
\item $(V,c)$ is of {\it diagonal type} if there exists a basis $x_1, \dots, x_{\theta}$ of $V$ and scalars $q_{ij}\in \kk^\times$ so that
$$c(x_i \otimes x_j) = q_{ij} (x_j \otimes x_i)$$
for all $1 \leq i,j \leq \theta$. The matrix $(q_{ij})$ is called the {\it braiding matrix}.
\item $(V,c)$ is of {\it finite Cartan type} if it is of diagonal type and 
\begin{equation} \label{eq:qij Cartan}
q_{ii} \neq 1 \text{ and } q_{ij}q_{ji} = q_{ii}^{a_{ij}}
\end{equation}
where $(a_{ij})_{1 \leq i,j \leq \theta}$ is a Cartan matrix associated to a semisimple Lie algebra.
\item The same terminology applies to a Hopf algebra $H$ when $gr(H) \cong \mf{B}(V) \#\kk G$.
\end{itemize}
\end{definition}

 Many examples of finite dimensional pointed algebras of finite Cartan type are provided throughout Section~\ref{sec:examples}; refer to the tables in Theorem~\ref{thm:main} for a summary.

\subsection{Quantum groups at roots of unity}\label{rootof1}

Let us recall  facts about quantum groups at roots of unity, which are examples of pointed Hopf algebras of finite Cartan type.
Consider the following notation. Let ${\mathfrak{g}}$ be a finite dimensional simple Lie algebra over $\kk$
of rank $r$ with Cartan matrix $(a_{ij})$ for $i,j=1,...,r$. Let $d_i$, for $i=1,...,r$, be relatively prime integers so that the matrix $(d_ia_{ij})$ is symmetric and
 positive definite. Let $q$ be an indeterminate.

Now consider the following Hopf algebra.

\begin{definition} \label{def:DJform} \cite{Drinfeld85} \cite{Drinfeld86} \cite{Jimbo} The Hopf algebra ${\mathcal U}_q(\mathfrak{g})$, referred to as {\it the Drinfeld-Jimbo quantum group attached to ${\mathfrak{g}}$}, is generated over $\kk[q,q^{-1}]$ 
by grouplike elements $k_i$, $(k_i,1)$-skew primitive elements $e_i$, and
$(1,k_i^{-1})$-skew primitive elements $f_i$, for $i =1, \dots, r$, with defining relations:
\[ 
\begin{array}{c} 
\vspace{.1in} 
k_ie_jk_i^{-1} =
q^{d_ia_{ij}} e_j,\quad \quad k_if_jk_i^{-1} = q^{-d_ia_{ij}} f_j, \quad \quad e_if_j - f_je_i = \delta_{ij}
\displaystyle \frac{k_i - k_i^{-1}}{q^{d_i}-q^{-d_i}},\\  
k_ik_j = k_jk_i, \quad \quad k_ik_i^{-1} = k_i^{-1}k_i =1,\\ 
\displaystyle \sum_{p=0}^{1-a_{ij}} (-1)^p {\small 
\begin{bmatrix} 1-a_{ij}\\p\end{bmatrix}_{q^{d_i}} }
e_i^{1-a_{ij}-p}e_je_i^p ~=~0 \quad \text{ for $i \neq j$},\\
\displaystyle \sum_{p=0}^{1-a_{ij}} (-1)^p {\small
  \begin{bmatrix}1-a_{ij}\\p\end{bmatrix}_{q^{d_i}} }
f_i^{1-a_{ij}-p}f_jf_i^p~=~0 \quad \text{ for $i \neq j$}.
\end{array} 
\] 
Here, ${\small \begin{bmatrix} n \\i \end{bmatrix}_q }=
\displaystyle \frac{[n]_q !}{[i]_q![n-i]_q!}$, where $[n]_q =
\displaystyle \frac{q^n - q^{-n}}{q-q^{-1}}$.
\end{definition}

\medskip

At roots of unity, the Hopf algebra  ${\mathcal U}_q(\mathfrak{g})$ has a finite dimensional quotient, which is defined below in Definition~\ref{def:FLkernel}.
To proceed, we must define root vectors, which was done by Lusztig   using the braid group action on  ${\mathcal U}_q(\mathfrak{g})$ \cite[Theorems~3.2,~6.6(ii) and Section~4.1]{Lusztig:rootof1}. 
Fix a reduced decomposition $S$ of the maximal element $w_0$ of the Weyl group of $W$:
 $w_0=s_{i_\ell} \dots s_{i_1}$. To this decomposition, there corresponds a normal ordering of positive roots: 
$\alpha_S^{(1)}=\alpha_{i_1}$,  $\alpha_S^{(2)}=s_{i_1}(\alpha_{i_2})$, $\dots,$ \linebreak $\alpha_S^{(i_\ell)}=s_{i_1}\cdots s_{i_{\ell-1}}(\alpha_{i_\ell})$ \cite{Zhelobenko}. 
It is known that every positive root occurs in this sequence exactly once, and a root $\alpha+\beta$ always occurs between $\alpha$ and $\beta$. 
So, given a positive root $\alpha$, let $N$ be the unique number such that $\alpha=\alpha_S^{(N)}$, and let $w_\alpha^S=s_{i_{N-1}}\cdots s_{i_1}$
so that $\alpha=w_\alpha^S(\alpha_{i_N})$. 
Define the root vectors by the formula 
$$e_\alpha^S:=T_{w_\alpha^S}(e_{i_N}) \quad \text{and} \quad
f_\alpha^S:=T_{w_\alpha^S}(f_{i_N}),$$
where, for a Weyl group element $w$, $T_w$ is the corresponding element of the braid group. Here, if $\alpha=\alpha_i$ is a
simple root, then $e_\alpha=e_i$ and $f_\alpha=f_i$.

Now we specialize to a root of unity.
\begin{equation} \label{eq:order}
\text{ Let $q \in \kk$ be a root of unity of odd order $m \geq 3$, with $m>3$ for type G$_2$.}
\end{equation}

\begin{definition} \label{def:DCK} \cite[Section~1.5]{deConciniKac} The Hopf algebra $\mathcal{U}_q(\mathfrak{g})$ specialized to a root of unity $q$ as in \eqref{eq:order} is known as the {\it Kac-de Concini quantum group of $\mf{g}$}.
\end{definition}

The desired finite dimensional quotient of  ${\mathcal{U}}_q({\mathfrak{g}})$ is now given as follows.

\begin{definition} \label{def:FLkernel} \cite[Section~5.7]{Lusztig:FinDim}
Take $q \in \kk$ satisfying \eqref{eq:order}. There exists a finite dimensional Hopf quotient of  ${\mathcal U}_q(\mathfrak{g})$  called 
 {\it the small quantum group} (or {\it the Frobenius-Lusztig kernel})
 {\it attached to 
 ${\mathfrak{g}}$}, denoted by $u_q({\mathfrak{g}})$. Namely, $u_q(\mf{g}) = {\mathcal U}_q(\mathfrak{g})/I$, where the Hopf ideal $I$ is generated by 
\begin{itemize}
\item $k_i^m =1$, for $i = 1, \dots, r$, and
\item (nilpotency relations)  $(e_\alpha^S)^m= (f_\alpha^S)^m=0$, for all positive roots $\alpha$.
\end{itemize}
 \end{definition}

 Even though the elements $e_\alpha^S$, $f_\alpha^S$ 
depend on $S$, the ideal $I$ is independent 
of the choice of $S$ \cite[Theorem~3.2]{Lusztig:rootof1}.  It is also known that the elements $(e_\alpha^S)^m$ and $(f_{\alpha}^S)^m$, along with $k_i^m$, are central in ${\mathcal{U}}_q({\mathfrak{g}})$ \cite[Corollary~3.1(a)]{deConciniKac}. The Hopf algebra $u_q({\mathfrak{g}})$
is a finite dimensional pointed Hopf algebra of dimension $m^{\dim{\mathfrak{g}}}$.

The Hopf algebra $u_q({\mathfrak{g}})$ has Hopf subalgebras $u_q^{\ge 0}({\mathfrak{g}})$, 
$u_q^{\le 0}({\mathfrak{g}})$, generated by the $k_i,e_i$ and the $k_i,f_i$, respectively, 
and subalgebras $u_q^+({\mathfrak{g}})$, 
$u_q^-({\mathfrak{g}})$, generated by the $e_i$ and by the $f_i$, respectively.
They are quotients of the corresponding subalgebras 
${\mathcal{U}}_q^{\ge 0}({\mathfrak{g}})$, ${\mathcal{U}}_q^{\le 0}({\mathfrak{g}})$,
${\mathcal{U}}_q^+({\mathfrak{g}})$, 
${\mathcal{U}}_q^-({\mathfrak{g}})$ of
${\mathcal{U}}_q({\mathfrak{g}})$, respectively.
 
 \begin{remark} \cite{Lusztig:rootof1} 
It is known that $u_q(\mf{g})$ is the finite dimensional 
Hopf subalgebra generated by $e_i,f_i,k_i$ inside Lusztig's ``big'' quantum 
enveloping algebra with divided powers, $U_q({\mathfrak{g}})$,
specialized to the root of unity. In fact, one has an exact sequence of Hopf algebras 
$u_q({\mathfrak{g}})\to U_q({\mathfrak{g}})\to U({\mathfrak{g}})$, 
where the second map is the quantum Frobenius map \cite[Section~8]{Lusztig:rootof1}. This is why $u_q(\mf{g})$ is also referred to as the Frobenius-Lusztig kernel.
 \end{remark}

\subsection{Hopf algebra actions} \label{subsec:actions}

We recall basic facts about Hopf algebra actions; refer to
\cite{Montgomery} for further details.  
A left $H$-module $M$ has a left $H$-action structure map denoted by $\cdot : H \otimes M \rightarrow M$.

\begin{definition} \label{def:Hopfact} Given a Hopf algebra $H$ and an algebra $A$, we say
that {\it $H$ acts on $A$} (from the left) if $A$ is a left $H$-module,
$h \cdot (ab) = \sum (h_1 \cdot a)(h_2 \cdot b)$,
and $h \cdot 1_A = \varepsilon(h)
1_A$ for all $h \in H$,  $a,b \in A$. Here, $\Delta(h) =
\sum h_1 \otimes h_2$ (Sweedler notation). In this case, we also say that $A$ is a {\it left $H$-module algebra}.

In the case that $H$ acts on a field $L$, we refer to $L$ as an {\it $H$-module field}.

\end{definition}

We restrict ourselves to $H$-actions that do not factor
through `smaller' Hopf algebras.

\begin{definition} \label{def:infaith} Given a left $H$-module $M$, we say that $M$ is an {\it
inner faithful} $H$-module if $IM\neq 0$ for every nonzero Hopf ideal
$I$ of $H$. Given an action of a Hopf algebra $H$ on an algebra $A$, we say that this
 action is  {\it inner faithful} if the left $H$-module  algebra $A$ is inner
faithful.
\end{definition}

When given an $H$-action on $A$, one can always pass uniquely to an inner faithful $\bar{H}$-action on $A$, where $\bar{H}$ is some quotient Hopf algebra of $H$. 

We also consider elements of $A$ invariant under the $H$-action on $A$.

\begin{definition}
Let $H$ be a Hopf algebra that acts on a $\kk$-algebra $A$ from the left. The {\it subalgebra of invariants} for this action is given by
$$A^H = \{ a \in A ~|~ h \cdot a  = \varepsilon(h) a~~ \text{ for all } h \in H\}.$$
\end{definition}

\subsection{Twists of Hopf algebras and of $H$-module algebras} \label{subsec:twist}
Let $J= \sum J^1 \otimes J^2$ be an invertible element in $H \otimes H$. Then, $J$ is a {\it Drinfeld twist} for $H$ if 
\begin{itemize}
\item $[(\Delta \otimes id)(J)](J \otimes 1) = [(\id \otimes \Delta)(J)](1 \otimes J)$, and
\item $(id \otimes \varepsilon)(J) = (\varepsilon \otimes id)(J) = 1.$
\end{itemize}

\begin{definition}{} 
(1) The Hopf algebra $H^J$ is a {\it Drinfeld twist of $H$ with respect to $J$} if $H^J = H$ as an algebra and  $H^J$ has the same counit as $H$ and coproduct and antipode given by
$$\Delta^J(h)= J^{-1} \Delta(h) J \quad \text{and} \quad S^J(h) = Q^{-1}S(h)Q,$$
where $Q = m(S \otimes id)J$, for all $h \in H$.

(2) Let $A$ be a left $H$-module algebra. Then, the {\it twisted algebra} $A_J$ has the same underlying vector space as $A$, and for $a, b \in A$, the multiplication of $A_J$ is given by
$$a \ast_J b = \sum (J^1 \cdot a )(J^2 \cdot b).$$
\end{definition}

Note that $J^{-1}$ is a twist for $H^J$, and $(H^J)^{J^{-1}}\cong H$. 
Also, if $A$ is an inner faithful left $H$-module algebra, then $A_J$ is an inner faithful left $H^J$-module algebra by using the same action of $H$ on the underlying vector space of $A$, and $(A_J)_{J^{-1}} \cong A$ as $H$-module algebras.

As discussed in \cite[page~799]{GKM}, Drinfeld twists $J$ have a special form when $H = \kk G$ for $G$ finite abelian.  For any $\chi \in \widehat{G}$, let ${\mathbf 1}_{\chi}$ be the idempotent $\frac{1}{|G|}\sum_{g \in G} \chi(g^{-1})g$ in $\kk G$. Then, $J = \sum_{\chi, \psi \in \widehat{G}} \sigma_J(\chi, \psi) {\mathbf 1}_{\chi} \otimes {\mathbf 1}_{\psi}$, for $\sigma_J$ a two-cocycle on $\widehat{G}$ with values in $\Bbbk^\times$. 

We also get an alternating bicharacter $b_J: \widehat{G} \times
\widehat{G} \rightarrow \kk^{\times}$ arising from $J$ given by 
$b_J(\chi, \psi) = \sigma_J(\psi, \chi)/\sigma_J(\chi,\psi)$ 
for all $\chi, \psi \in \widehat{G}$. 

\begin{proposition} \cite[pages~798-799]{GKM} \label{prop:JtobJ}
The assignment $J \mapsto b_J$ is a bijection between gauge equivalence classes of Drinfeld twists and alternating bicharacters. \qed
\end{proposition}

Now we have the following result for twisted polynomial rings.

\begin{theorem} \label{thm:GKM} \cite[Theorem~3.8]{GKM} Let $G$ be an abelian group and let $A=\kk[z_1, \dots z_n]$ be a polynomial ring with a $G$-action such that $z_i$ are common eigenvectors of $G$. Let $\chi_i$ be the character of $G$ corresponding to the $G$-action on the eigenvector $z_i$, that is to say,  $g \cdot z_i = \chi_i(g)z_i$. Then, the twisted algebra $A_J$ has generators $z_i$ with defining relations: $$z_i \ast_J z_j = b_J(\chi_j, \chi_i) z_j \ast_J z_i.$$

\vspace{-.21in}
 \qed
\end{theorem}


\section{Galois-theoretical Hopf algebras} \label{sec:prelim}

We begin by motivating the notion of a {\it Galois theoretical} Hopf algebra, or a Hopf algebra $H$ that acts inner faithfully on a field $L$. To this end, recall that our goal is to classify inner faithful actions of certain Hopf algebras on commutative domains.

\begin{lemma} \label{lem:Skryabin} Let $A$ be a commutative domain and $Q_A$ be its quotient field. Namely, $Q_A = A\mc{S}^{-1}$, for the set $\mc{S}$ of nonzero elements of $A$. If a finite dimensional Hopf algebra $H$ acts on $A$ inner faithfully, then the action of $H$ on $A$ extends to an inner faithful action of $H$ on $Q_A$. 
\end{lemma}

\begin{proof} By \cite[Lemma~1.1]{Skryabin:invariants}, an inner faithful $H$-action on a commutative domain $A$ extends to an inner faithful $H$-action on the localization $A \tilde{\mc{S}}^{-1}=A\otimes_{A^H}A^H\tilde{\mc{S}}^{-1}$, for $\tilde{\mc{S}}$ a multiplicatively closed subset of $A^H$. Since $A$ is a commutative domain, we have by \cite[Theorem~2.5 and~Proposition~2.7]{Skryabin:invariants} that $A$ is integral over $A^H$. (Here, $A$ is {\it $H$-reduced}, as $A$ is a domain.) Now, take $\mc{S}$ to be the set of nonzero elements of $A$, and we get that the $H$-action on $A$ extends naturally  to an inner faithful $H$-action on the field of quotients 
$Q_A := A\mc{S}^{-1} \cong A \otimes_{A^H} A^H\tilde{\mc{S}}^{-1}.$
\end{proof}

\begin{remark} \label{rem:liftcomdomain}
Conversely, any inner faithful $H$-action on a field $L$ yields an inner faithful $H$-action on a finitely generated commutative domain $A$. To see this, pick a finite dimensional $H$-submodule $V$ of $L$ which generates Rep$H$ as a tensor category, which exists due to inner faithfulness. Take $A$ to be generated by $V$ inside $L$. Then, $H$ acts on $A$. This shows that there is always a finitely generated domain $A \subset L$ that is $H$-stable and has an inner faithful action of $H$.
\end{remark}

Thus, we consider Hopf algebra actions on fields for the remainder of this work.
 Let us introduce the following terminology.

\begin{definition} \label{def:GT}
A Hopf algebra $H$ over $\kk$ is said to be {\it Galois-theoretical} if it  acts inner faithfully and $\kk$-linearly on a field containing $\kk$.
\end{definition}

Note that if a Galois-theoretical Hopf algebra $H$, say with $H$-module field $L$, yields an $H^*$-Galois extension $L^H \subset L$, then $H$ is a group algebra. However, the Hopf actions in this work do not yield Hopf-Galois extensions in general as $H$ is noncocommutative. Basic results about Galois-theoretical Hopf algebras are collected in the proposition below.  

\begin{proposition}  \label{prop:GTprelim}
We have the following statements.
\begin{enumerate}
\item Any finite group algebra is Galois-theoretical.
\item Any semisimple Galois-theoretical Hopf algebra is a group algebra.
\item The restriction of an inner faithful action of a Hopf algebra to a Hopf subalgebra is inner faithful. 
In particular, a Hopf subalgebra of a Galois-theoretical Hopf algebra is Galois-theoretical. 
\item Any finite dimensional Galois-theoretical Hopf algebra whose coradical is a Hopf subalgebra is pointed.
\item  If $H$ and $H'$ are Galois-theoretical Hopf algebras,  then so is $H\otimes H'$.
\item If $H$ is Galois-theoretical, then $\kk S_n \ltimes H^{\otimes n}$ is Galois-theoretical for  all $n \geq 1$.
\end{enumerate}
\end{proposition}

\begin{proof}
(a) It is well known that any finite group can be realized as a Galois group of a field extension.

(b) This follows from \cite[Theorem~1.3]{EtingofWalton:ssHopf}.

(c) Let $H$ act inner faithfully on a module $M$ 
and let $H'\subset H$ be a Hopf subalgebra. Let $I$ be a Hopf ideal of $H'$ annihilating $M$. 
Let $J=HIH$. Then, $J$ is a Hopf ideal in $H$ annihilating $M$, so $J=0$ and hence, $I=0$.

(d) The coradical $H_0$ of $H$ is cosemisimple, and thus, semisimple by \cite{LarsonRadford}. So, $H_0 = \kk G(H)$ by (b) and (c). Hence, $H$ is pointed by Lemma~\ref{lem:pointed}(a).

(e) If $H$ acts on a field $L$ inner faithfully and $H'$ acts on a field $L'$ inner faithfully, then $H\otimes H'$ acts on the quotient field of $L\otimes L'$ inner faithfully. 

(f)  First, we need the result below.
\medskip

\noindent {\bf Lemma}.  We have the following statements.
\begin{enumerate}
\item[(i)] Let $B$ be an associative algebra over $\Bbbk$, and $V$ be a $B$-module containing vectors $v_1, \dots ,v_n$ linearly independent over $B$ (that is to say, $V$ contains $B^n$ as a submodule).
Then, $V^{\otimes n}$ is a faithful module over $\Bbbk S_n\ltimes B^{\otimes n}$ (where $S_n$ acts on $B^{\otimes n}$ by permutation of components).

\item[(ii)] Take $B$ to be a finite dimensional associative algebra over $\kk$. If $W$ is a faithful $B$-module and $V=W\otimes X$, with $X$  an infinite dimensional $\Bbbk$-vector space,
then $\Bbbk S_n\ltimes B^{\otimes n}$ acts faithfully on $V^{\otimes n}$ for any $n$.

\item[(iii)] If a finite dimensional Hopf algebra $H$ acts inner faithfully on an algebra $A$, then $H$ acts faithfully on $A^{\otimes s}$ for some $s$.
\end{enumerate}
\medskip

\noindent {\it Proof of Lemma}. (i) Consider the map $f: \Bbbk S_n\ltimes B^{\otimes n}\to V^{\otimes n}$ given by $f(x)=x \cdot (v_1\otimes \dots \otimes v_n)$.
Since the map $b\mapsto bv_i$ defines an isomorphism $B\to Bv_i$, and the sum $Bv_1+\dots+Bv_n$ is direct, we see that
$f$ is injective, which implies (i). 

(ii) Since $W^m$ contains a copy of $B$ for some $m$, we have that $V$ contains $B^n$ for any $n$. Now statement (ii) follows from (i). 

(iii) Let $K_{s} \subset H$ be the kernel of the action of $H$ on $A^{\otimes s}$. Observe that $K_{s} \supset K_{s+1}$ because $A^{\otimes s}=A^{\otimes s}\otimes 1\subset A^{\otimes s+1}$. Let $K=\bigcap_{s} K_{s}$. There is an integer $s_0$ such that $K=K_s$ for all $s \ge s_0$. Given $h \in K$, consider the action of $\Delta(h)$ on $A^{\otimes s}\otimes A^{\otimes t}$ for $s,t\geq s_0$. Since $A^{\otimes s}\otimes A^{\otimes t}$ is a faithful module
over $H/K\otimes H/K$, we find that $\Delta(h)\in K\otimes H+H\otimes K$. Thus, $K$ is a bialgebra
ideal of $H$, hence a Hopf ideal as $H$ is finite dimensional. Since $H$ acts on $A$ inner faithfully, this implies that $K=0$,
as claimed.
\qed
\medskip

Now we verify part (f) of the proposition above. Fix a commutative domain $A$ over $\Bbbk$ that admits an inner faithful action of $H$. Then, $H$ acts faithfully on the space $W:=A^{\otimes s}$ for some $s$ by part (iii) of the Lemma.
So applying part (ii) of the Lemma to $X=\Bbbk[x_1,\dots,x_s]$, we conclude that $S_n\ltimes H^{\otimes n}$ acts faithfully on $$(A^{\otimes s}[x_1,\dots,x_s])^{\otimes n}=
(A[x]^{\otimes s})^{\otimes n}=(A[x]^{\otimes n})^{\otimes s}.$$ This means that $S_n\ltimes H^{\otimes n}$ acts inner faithfully on the
commutative domain \linebreak $A[x]^{\otimes n}$, where $H$ acts trivially on $x$. Thus, $S_n\ltimes H^{\otimes n}$ is Galois-theoretical.
\end{proof}

\begin{question}\label{que1}
(a) If a finite group $\Gamma$ acts on a Galois-theoretical Hopf algebra $H$, then is $\Gamma \ltimes H$ Galois-theoretical? If true, then this would be a generalization of Proposition~\ref{prop:GTprelim}(f).

(b) Is a Hopf algebra quotient of a Galois-theoretical Hopf algebra $H$ also
Galois-theoretical?
For example, if $c$ is a central grouplike element of $H$, is then $H/(c-1)$
Galois-theoretical? In particular, if $L$ is an inner faithful
$H$-module field, is then $L^c$ always an inner faithful $H/(c-1)$-module field? 
\end{question}

Along with Proposition~\ref{prop:GTprelim}(f), special cases of Question \ref{que1}(a) has been addressed in Propositions~\ref{prop:E(n)skewprod} and~\ref{prop:H81skewprod}.

Now we provide a general result about invariants
of pointed Hopf algebra actions on commutative domains. 

\begin{theorem}\label{inva}
(i) Let $H$ be a finite dimensional pointed Hopf algebra over
$\Bbbk$ with $G(H)=G$, and assume that $H$ acts on a
commutative domain $A$. Then, $A^H=A^G$. 

(ii) If in the situation of (i), $A=L$ is a field, and $H$ acts inner faithfully on $L$,
then the field extension $L^H=L^G\subset L$ is a finite
Galois extension with Galois group $G$. 
\end{theorem}

\begin{proof}
(i) We prove by induction in $n$ that if $x\in H_n$, 
and $\varepsilon(x)=0$
then $x$ acts by zero on $A^G$, which implies the required statement. 
For $n=0$, this is tautological as $H_0=\Bbbk G$.
So let us assume that $n>0$ and that the statement is known 
for $n-1$. By the Taft-Wilson theorem (see 
\cite[Theorem~5.4.1]{Montgomery}), we may assume without loss of generality that
$$
\Delta(x)=g\otimes x+x\otimes g'+w,
$$
where $g,g'\in G$, $w\in H_{n-1}\otimes H_{n-1}$
and $(\varepsilon\otimes \varepsilon)(w)=0$
(as $H_n/H_{n-1}$ is spanned by such elements $x$). Let $f_1,f_2\in A^G$. 
Using the induction assumption, we have that
$$
x \cdot(f_1 f_2) = (g \cdot f_1)(x \cdot f_2) +(x \cdot f_1)(g' \cdot f_2) + w \cdot (f_1f_2) = f_1 (x \cdot f_2) + (x \cdot f_1) f_2.
$$
Thus, $x: A^G\to A$ is a derivation.

On the other hand, since $H$ is finite dimensional, 
by Skryabin's theorem 
(\cite{Skryabin:invariants}, Theorem 6.2(iii)), 
$A$ is integral over the subalgebra of invariants $A^H$. 
Thus, so is $A^G$. Hence, the equality $x|_{A^G}=0$ follows from the
following well known lemma from commutative algebra.

\begin{lemma}\label{deri} Let $B\subset C$ be an integral
  extension of commutative domains, $M$ be a torsion-free $C$-module, 
and suppose that $x: C\to M$ is a
  derivation such that $x|_B=0$. Then, $x=0$. 
\end{lemma}

\noindent {\it Proof of Lemma~\ref{deri}}.
For $c \in C$, consider the minimal monic polynomial of $c$
over $B$, 
$$
p(c) = c^n + b_{n-1} c^{n-1} + \dots + b_1 c +
b_0,
$$ 
with $b_i \in B$, which exists since $C$ is integral over $B$. Letting $x$ act on the equation $p(c)=0$, we have that
$$
[nc^{n-1} + (n-1)b_{n-1}c^{n-2}+ \cdots + b_1](x \cdot c) = 0.
$$ 
The first factor of the left hand side (the derivative $p'(c)$) is not equal to zero
due to the minimality of 
$p(c)$ and the fact that $n\ne 0$ (as we are in characteristic
zero). Thus, since $M$ is a torsion-free $C$-module, we have 
$x \cdot c=0$ for all $c \in C$, as desired.
\qed

Returning to the proof of Theorem~\ref{inva}, we see that the proof of (i) is completed by applying Lemma
\ref{deri} to $B=A^H$,
$C=A^G$, $M=A$. 

(ii) This follows from (i), as clearly the group $G$ must act
faithfully on $L$. 
\end{proof} 

\begin{corollary} \label{cor:extnGalois}
Let $H$ be a Hopf algebra (not necessarily finite dimensional) 
generated by a finite group of grouplike elements 
$G=G(H)$ and a set of $(g_i,1)$-skew primitive elements $x_i$ 
for some $g_i\in G$. Assume that for each $i$, the Hopf
subalgebra generated by $\{g_i, x_i\}$ is finite dimensional. 
Then:

(i) We have that $A^H = A^G$ for any commutative domain $A$ that arises as an $H$-module algebra.

(ii) If $H$ acts inner faithfully on a field $L$, then the field
  extension $L^H=L^G \subset L$ is Galois with Galois group $G$.
\end{corollary}

\begin{proof}
By Theorem \ref{inva}, $x_i$ acts by zero on $A^{g_i}$, hence
on $A^G$. This implies both statements. 
\end{proof}

Thus, when $H$ is Galois-theoretical and generated by grouplike
and skew primitive elements, the field extensions that arise as
$H$-module fields may be understood in terms of classical Galois
theory. This phenomenon is illustrated in several examples in the
next section, particularly when $G(H)$ is a cyclic group and $L^H
\subset L$ is a cyclic extension.

We also have the following generalization of Theorem~\ref{inva}.

\begin{theorem}\label{inva2} 
Let $H$ be a finite dimensional Hopf algebra over
$\Bbbk$. If $H$ acts on a
commutative domain $A$, then $A^H=A^{H_0}$
(even if $H_0$ is not a subalgebra). 
\end{theorem}

\begin{proof} 
As before, we show by induction in $n$ that 
$x\in H_n$ with $\varepsilon(x)=0$ acts by zero on
$A^{H_0}$. It is shown similarly to the Taft-Wilson theorem
that $H_n/H_{n-1}$ is spanned by elements 
$x_{ii',C,C'}$, where $C,C'$ are simple subcoalgebras 
of $H$, and
$$
\Delta(x_{ii',C,C'})=\sum_j t_{ij}\otimes
x_{ji',C,C'}+\sum_{j'}x_{ij',C,C'}\otimes t'_{j'i'}+w,
$$
where $t_{ij}$ is a basis 
of $C$ such that $\Delta(t_{ij})=\sum_k t_{ik}\otimes t_{kj}$ and
$t'_{i'j'}$ is a similar basis of $C'$. Moreover, 
$w\in H_{n-1}\otimes H_{n-1}$
 is such that $(\varepsilon\otimes \varepsilon)(w)=0$.
So without loss of generality we may assume that
$x=x_{ii',C,C'}$. Then by the induction assumption, $x$ is a derivation of 
$A^{H_0}$ into $A$. The rest of the proof is the same as that of
Theorem \ref{inva}(i). 
\end{proof}

Even though this paper is about actions of Hopf algebras on
commutative algebras, let us give a generalization of Theorems
\ref{inva} and \ref{inva2} to the noncommutative case. Namely, we provide a result for Hopf actions on Azumaya algebras. Recall that examples of Azumaya algebras include matrix algebras over commutative algebras and central simple algebras.

\begin{theorem}\label{inva3} 
Let $H$ be a finite dimensional Hopf algebra over
$\Bbbk$.

(i) Assume that $H$ acts on an Azumaya algebra $A$ with center $Z$, where $Z$ is an integral domain. 
Let $Z^H=Z\cap A^H$ and $Z^{H_0}=Z\cap A^{H_0}$. Then, $Z^H=Z^{H_0}$. 

(ii) If, in addition to the hypotheses of (i), $H$ is pointed, then $Z^H = Z^G$ for $G = G(H)$.
\end{theorem}

\begin{proof} 
(i) As in the proof of Theorem~\ref{inva2}, we get that $x$ defines a derivation from
$Z^{H_0}$ to $A$. By
\cite[Theorem~3.1(ii)]{Etingof:PIintegrality}, $A$ is integral over $Z^H$. Hence, $Z^{H_0}$ is also integral
over $Z^H$, (i.e., $Z^{H_0}$ is an algebraic field extension of $Z^H$). So the statement follows from Lemma \ref{deri},
specialized to $B=Z^H$, $C=Z^{H_0}$, and $M=A$. 

(ii) This follows immediately from part (i) and Lemma~\ref{lem:pointed}(a).
\end{proof}

 
\section{Examples and non-examples of Galois-theoretical Hopf algebras}
\label{sec:examples}

In this section, we study examples and non-examples of finite dimensional pointed Galois-theoretical Hopf algebras, including 

\begin{tabular}{lll}
$\bullet$ Taft algebras $T(n)$& {[type A$_1$]} & (Section~\ref{subsec:T(n)}),\\
$\bullet$  Nichols Hopf algebras $E(n)$  & {[type A$_1^{\times n}$]}  &(Section~\ref{subsec:E(n)}),\\
$\bullet$  generalized Taft algebras $T(n,m,\alpha)$ &{[type A$_1$]}& (Section~\ref{subsec:genTaft}),\\
$\bullet$  book algebras $\mathbf{h}(\zeta,p)$ & {[type A$_1 \times$A$_1$]} & (Section~\ref{subsec:book}),\\
$\bullet$  the 81-dimensional Hopf algebra $H_{81}$ & {[type A$_2$]}  &(Section~\ref{subsec:H81}),\\
$\bullet$  $u_q(\mf{sl}_2)$ and gr$(u_q(\mf{sl}_2))$ & {[type A$_1 \times$A$_1$]}  & (Section~\ref{subsec:uqsl2}),\\
$\bullet$  $u_q(\mf{gl}_2)$ & {[type A$_1 \times$A$_1$]}  & (Section~\ref{subsec:uqslntwist}),\\
$\bullet$  some Drinfeld twists of $u_q(\mf{gl}_n)$, $u_q(\mf{sl}_n)$ & {[type A$_{n-1}\times$A$_{n-1}$]} &  (Section~\ref{subsec:uqslntwist}),\\
$\bullet$ some Drinfeld twists of $u_q^{\geq 0}(\mf{g})$ & {[same type as $\mf{g}$]}  & (Section~\ref{subsec:uqgtwist}).
\end{tabular}
\smallskip

Altogether, the propositions in these sections yield a proof of Theorem~\ref{thm:main}. An example of a Galois-theoretical Drinfeld twist of $u_{q^{1/2}}(\mf{gl}_2)$ is provided in Section~\ref{subsec:moduqgl2}. We also present a finite dimensional non-pointed Galois-theoretical Hopf algebra in Section~\ref{subsec:nonptedGT}. We end with a discussion of the Galois-theoretical property of duals and twists of Hopf algebras in Section~\ref{subsec:GTtwist}. 

To begin, consider the notation and the preliminary result provided below.
\medskip

\noindent {\it Notation}. 
 Let $L$ be a $\kk \mathbb{Z}_n$-module field, for $\mathbb{Z}_n = \langle g ~|~ g^n =1\rangle$. Let $\zeta$ be a primitive $n$-th root of unity. We set 
$L_{(i)} := \{r \in L ~|~ g \cdot r = \zeta^{-i} r\}$ for $i=0,\dots, n-1$.

\begin{lemma} \label{lem:Zn} Given an inner faithful $\Bbbk\mathbb{Z}_n$-module field $L$ as above, we have that:
\begin{enumerate}
\item $L$ is $\mathbb{Z}_n$-graded and decomposes as a direct sum of $g$-eigenspaces $L_{(i)}$ with eigenvalue $\zeta^{-i}$, where $L_{(0)} = L^{\mathbb{Z}_n}$
and $L_{(1)} \neq 0$.
\item For $u \in L_{(1)}$, we have that $L$ is an extension of $L^{\mathbb{Z}_n}$, so that $$L~=~L^{\mathbb{Z}_n}[u]/(u^n - v).$$
Here, $v$ is a non-$n'$-th power in $(L^{\mathbb{Z}_n})^{\times}$ for any $n'>1$ dividing $n$.
\end{enumerate}
\end{lemma}

\begin{proof}
Part (a) is clear. In particular, $L_{(1)} \neq 0$ due to inner faithfulness. Part (b) follows since $t^n -v$ is the minimal polynomial of the element $u$.
\end{proof}


\subsection{The Taft algebras $T(n)$ are Galois-theoretical} \label{subsec:T(n)} 
Take $n \geq 2$ and let  $\zeta$ a primitive $n$-th root of unity.
Let $T(n)$ be the {\it Taft algebra} of dimension $n^2$, which is generated by a grouplike element $g$ and a $(g,1)$-skew primitive element $x$, subject to relations 
$$g^n =1, \quad x^n = 0, \quad gx=\zeta xg.$$ 
We have that $T(n)$ acts inner faithfully on the commutative
domain $\kk[z]$ by
$$g \cdot z = \zeta^{-1} z , \quad x \cdot z =1.$$
So,  $T(n)$ is Galois-theoretical by Lemma~\ref{lem:Skryabin}. More explicitly, we can extend the action of $T(n)$ on $\kk[z]$ to an action of $T(n)$ on $\kk(z)$ since $T(n)$ acts trivially on $\kk[z^n]$ and $\kk(z) = \kk[z] \otimes_{\kk[z^n]} \kk(z^n)$. Further, we classify all inner faithful $T(n)$-module fields below, which recovers \cite[Theorem~2.5]{MontgomerySchneider}.

\begin{proposition} \label{prop:T(n)}
The Taft algebras $T(n)$ are Galois-theoretical, and the fields $L$ that admit an inner faithful $T(n)$-action are precisely of the form $$L = F[u]/(u^n-v)$$ for $F = L^{T(n)}$, $u \in L_{(1)}$, and $v$ a non-$n'$-th power in $F^{\times}$, for any $n'>1$ dividing $n$. 
So, $L$ is a cyclic degree $n$ Galois extension of its subfield of invariants $F$ with Galois group $\mathbb{Z}_n$.
We also have that
$g \cdot u = \zeta^{-1} u,  ~x \cdot u =1$ and $g \cdot r_0 = r_0, ~ x \cdot r_0 = 0$ for all $r_0 \in F$. \qed
\end{proposition}

\begin{proof}
Let us determine the $T(n)$-module fields $L$.  Since $G(T(n)) \cong  \mathbb{Z}_n$, we can employ Lemma~\ref{lem:Zn}. Observe that $L^{\mathbb{Z}_n} = L^{T(n)}$ by Theorem~\ref{inva}(i); let us denote this field by $F$. Take a nonzero element $u \in L_{(1)}$.
Since $g \cdot (x \cdot u) ~=~ \zeta x \cdot (g \cdot u) ~=~ x \cdot u$, we have that $x \cdot u = w \in F$. Moreover, we can replace $u$ with $w^{-1}u$ to get that $x \cdot u =1$.   
Also, $x \cdot r_0=\varepsilon(x) r_0= 0$ for all $r_0 \in F$. 
Finally, the Galois group of the extension $L^{T(n)} \subset L$ is $G(T(n)) = \mathbb{Z}_n$ by Theorem~\ref{inva}(ii).
\end{proof}

One can reformulate Proposition~\ref{prop:T(n)} as follows. 
 
\begin{proposition} \label{prop:T(n) reform}
Fields $L \supset \kk$ with an inner faithful $T(n)$-action are in one-to-one correspondence with fields $F \supset \kk$ together with  a non-$n'$-th power $v \in F^{\times}$, for any $n'>1$ dividing $n$.
\end{proposition}

\begin{proof}
Retain the notation in Lemma~\ref{lem:Zn} and Proposition~\ref{prop:T(n)}. So, we have a field $L \supset \kk$ with an inner faithful $T(n)$-action if and only if 
$L=F[u]/(u^n-v)$, where $t^n-v \in F[t]$ is the minimal polynomial of $u \in L_{(1)}$ over $F = L^{\mathbb{Z}_n}$. For $L$ to be a field, this polynomial must be irreducible.  
So it remains to show that the polynomial $t^n-v$ is irreducible if and only if $v$ is a non-$n'$-th power in $F^{\times}$, for any $n'>1$ dividing $n$; see, for instance, \cite[Chapter~5, Section~11.8, Example~4]{BourbakiAlgII}.

The forward direction of this claim is clear. Conversely, 
suppose that $v\in F^\times$ and an irreducible polynomial $p(t)=t^s+q(t)$ divides $t^n-v$, with $\deg q(t) <s<n$. The group $\mathbb{Z}_n$ of roots of unity of order $n$
acts on such divisors by $p(t) \mapsto \zeta^{-s}p(\zeta t)$, where $\zeta$ is any $n$-th root of unity. Clearly, the stabilizer of $p(t)$ 
is contained in $\mathbb{Z}_s$ (as the constant term of $q(t)$ is nonzero). 
So, it must be exactly $\mathbb{Z}_s$. Else, there will be more than $n/s$ distinct monic irreducible divisors of $t^n-v$ of degree $s$, 
and their product must divide $t^n-v$, which is a contradiction. This means that $p(t)$ cannot contain any terms other than $t^s$ and constant term, that is to say,  $p(t)=t^s-f$ for $f \in F$. 
Hence, $n/s$ is an integer, and $f^{n/s}=v$. Thus,  the reverse direction of the claim holds.
\end{proof}


\subsection{The Nichols Hopf algebras $E(n)$ are Galois-theoretical} \label{subsec:E(n)} 
Take $n \geq 1$.
Let $E(n)$ be the Nichols Hopf algebra of dimension $2^{n+1}$, generated by a grouplike element $g$ and $(g,1)$-skew primitive elements $x_1, \dots, x_n$, subject to relations 
$$g^2 =1, \quad x_i^2 = 0, \quad g x_i=-x_i g, \quad x_ix_j = -x_jx_i.$$ 
We have that $E(n)$ acts inner faithfully on the commutative domain $\kk[z]$ and field $\kk(z)$ by 
$$g \cdot z = - z , \quad x_i \cdot z =z^{2(i-1)}.$$
  One sees this as $x_i \cdot z^r = 0$ for all $i$ and $r$ even. Thus, $E(n)$ is  Galois-theoretical by Lemma~\ref{lem:Skryabin}. By a similar argument to that in Section~\ref{subsec:T(n)}, $\kk(z)$ is an inner faithful $E(n)$-module field.

To determine all inner faithful $E(n)$-module fields $L$, observe that $G(E(n)) = \mathbb{Z}_2$ and use an argument similar to that in Section~\ref{subsec:T(n)} to get the following result.

\begin{proposition} \label{prop:E(n)}
The Hopf algebras $E(n)$ are Galois-theoretical and the fields $L$ that admit an inner faithful $E(n)$-action are precisely of the form $$L = F[u]/(u^2-v)$$ for $F  = L^{E(n)}$, $u \in L_{(1)}$, and $v$ a nonsquare element of $F^{\times}$. So, $L$ is a quadratic Galois extension of its subfield of invariants $F$ with Galois group $\mathbb{Z}_2$.
We have that
$g \cdot u = -u,  ~x_i \cdot u = w_i \in F$ for $\{w_i\}_{i=1, \dots n}$ linearly independent over $\kk$, and $g \cdot r_0 = r_0, ~ x \cdot r_0 = 0$ for all $r_0 \in F$.
\end{proposition}

\begin{proof}
It suffices to establish inner faithfulness. Note that any nonzero Hopf ideal of $E(n)$ has nonzero intersection with span$_{\kk}(x_1,...,x_n)$ \cite[Corollary~5.4.7]{Montgomery}. So if $\{w_i\}_{i=1, \dots n}$ are linearly independent,  then $\{x_i\}_{i=1, \dots n}$ act by linearly independent linear transformations of $L$. 
Thus, the action is inner faithful.
\end{proof}

Note that while $E(n)$ can act inner faithfully on a field, it follows from the result above that $E(n)$ cannot act faithfully on a field (and hence, on a commutative domain). Indeed, the elements $gx_i-x_i$ act necessarily by zero for all $i$.

We also have the following generalization of the proposition above.

\begin{proposition} \label{prop:E(n)skewprod} Retain the notation above.
Let $G$ be a finite subgroup of $GL_n(\kk)$.  Then, one can form the semi-direct product
Hopf algebra $\kk G \ltimes E(n)$,  where $GL_n(\kk)$ acts on $E(n)$ by linear transformations of the skew primitive elements $x_i$ for $i=1, \dots n$.
Moreover, the Hopf algebra $\kk G\ltimes E(n)$ is Galois theoretical. 
\end{proposition}

\begin{proof}
For the first statement, note that one can check directly that the ideal of relations of $E(n)$ is stable under the action of $GL_n(\kk)$. For the last statement, proceed as follows. Let $F=\kk(w_1,...,w_n)$, where $\{w_i\}$ are algebraically independent, 
which has an action of $G$ via the embedding of $G$ into $GL_n(\kk)$. Pick a non-square element $v \in (F^G)^{\times}$. Consider the $E(n)$-module field $L=F[u]/(u^2-v)$. 
Then, the actions of $G$ and of $E(n)$ on $L$ combine into an inner faithful action of $\kk G\ltimes E(n)$ on $L$.
\end{proof}



\subsection{On the generalized Taft algebras $T(n, m,\alpha)$ being Galois-theoretical}
\label{subsec:genTaft} 

Let $\alpha \in \kk$ and let $n,m$ be positive integers so that $m$ divides $n$. Let $q$ be a primitive $m$-th root of unity.
Consider the {\it generalized Taft algebra} $T(n, m, \alpha)$, which is a Hopf algebra generated by a grouplike element $g$ and $(g,1)$-skew primitive element $x$, subject to the relations
$$g^n=1, \quad x^m = \alpha(g^m -1),  \quad \quad gx = q xg.$$
So, $T(n,n, 0) = T(n)$ is a Taft algebra; see Section~\ref{subsec:T(n)}. The Galois-theoretical property of $T(n,m,\alpha)$ is given as follows.

\begin{proposition} \label{prop:genTaft}
A generalized Taft algebra $T(n,m,\alpha)$ is Galois-theoretical if and only if $m=n$, that is to say, if and only if $T(n,m, \alpha)$ is a Taft algebra $T(n)$.
\end{proposition}

\begin{proof}
If $m=n$, then $T(n,m,\alpha)=T(n)$, and is Galois-theoretical by Proposition~\ref{prop:T(n)}.

On the other hand, suppose $T(n,m,\alpha)$ is Galois-theoretical with inner faithful module field $L$. Since $T(n,m,\alpha)$ is generated by grouplike and skew primitive elements, $L^{T(n,m,\alpha)}$ =  
$L^{G(T(n,m,\alpha))}$ = $L^{\mathbb{Z}_n}$ by Theorem~\ref{inva}; let us denote this field by $F$. Then, $\mathbb{Z}_n = \langle g \rangle$ acts faithfully on $L$. By Lemma~\ref{lem:Zn}, $L = \bigoplus_{i=0}^{n-1} L_{(i)}$, where we can take $g \cdot r = \zeta^i r$ for all $r \in L_{(i)}$, with $\zeta$ is a primitive $n$-th root of unity such that  $q=\zeta^{n/m}$. We also have by Lemma~\ref{lem:Zn} that $L_{(0)} = F$ and $L = F[u]/(u^n-v)$ for $u \in L_{(1)}$ and $v$ a non-$n'$-th power in $F^{\times}$, for any $n'>1$ dividing $n$.

By way of contradiction, suppose that $n/m = s >1$, so that $T(n,m,\alpha)$ is not a Taft algebra.  Since $$g \cdot (x \cdot u) ~=~ qx \cdot (g \cdot u) ~=~ q\zeta x \cdot u ~=~ 
\zeta^{s+1}x \cdot u,$$ we get that $x \cdot u \in L_{(s+1)}$.
Now $x \cdot u = r_0 u^{s+1}$ for some $r_0 \in F^{\times}$. Let $[d]$ denote $ \frac{1-\zeta^{d}}{1-\zeta}$. Then,
$x \cdot u^d = [d] r_0 u^{s+d}$.
Therefore,
\[
\begin{array}{rl}
x^m \cdot u &= x^{m-1} \cdot r_0 u^{s+1}\\
                   &= x^{m-2} \cdot \left(r_0 (x \cdot u^{s+1})) ~~=~~ x^{m-2} \cdot ([s+1] r_0^2 u^{2s+1}\right)\\
                  &=x^{m-3} \cdot \left([s+1] [2s+1] r_0^3 u^{3s+1}\right)\\
& \vdots\\
&= [s+1] [2s+1] \cdots [(m-1)s+1] r_0^m uv,
\end{array}
\]
with $[\ell s+1] \neq 0$ for $\ell = 1, \dots, m-1$. On the other hand, $\alpha(g^m-1) \cdot u = \alpha(\zeta^{m}-1) u$. Using the relation $x^m  = \alpha(g^m-1)$, we get that
$$v = \frac{\alpha(\zeta^{m}-1)}{[s+1] [2s+1] \cdots [(m-1)s+1]}r_0^{-m}.$$
This yields a contradiction as $v$ is a non-$m$-th power in $F^{\times}$. Thus, $m=n$ as required.
\end{proof}


\subsection{On the book algebras $\mathbf{h}(\zeta,p)$ being Galois-theoretical} \label{subsec:book}
Let $p<n$ be coprime positive integers with $n \geq 2$ and let $\zeta$ be a primitive $n$-th root of unity. The {\it book algebra} $\mathbf{h}(\zeta,p)$ is a Hopf algebra 
 generated by a grouplike element $g$, a $(1,g)$-skew primitive element $x_1$, and a $(g^p,1)$-skew primitive element $x_2$, subject to relations: 
$$g^n=1, \quad x_1^n =x_2^n =0, \quad gx_1=\zeta x_1g, \quad gx_2=\zeta^p x_2g, \quad x_1x_2=x_2x_1;$$
 see \cite[Introduction]{AndSch:p3}. The Galois-theoretical property of $\mathbf{h}(\zeta,p)$ is given as follows.

\begin{proposition} \label{prop:book}
A book  algebra $\mathbf{h}(\zeta,p)$ is Galois-theoretical if and only if $p=1$. In this case, any $\mathbf{h}(\zeta,1)$-module field $L$ is a cyclic degree $n$ Galois extension of its subfield of invariants $L^{\mathbf{h}(\zeta,1)}$ as in Lemma~\ref{lem:Zn}.
\end{proposition}

\begin{proof}
If $p=1$, then $\mathbf{h}(\zeta,p)$ is Galois-theoretical since it acts inner faithfully on the commutative domain $\kk[z]$ and field $\kk(z)$ by 
$$g \cdot z = \zeta^{-1} z, \quad x_1 \cdot z =1, \quad x_2 \cdot z =1.$$
To see inner faithfulness, note that any nonzero Hopf ideal of $\mathbf{h}(\zeta,1)$ contains either $x_1$ or $x_2$ \cite[Corollary~5.4.7]{Montgomery}.

Suppose $\mathbf{h}(\zeta,p)$ is Galois-theoretical with module field $L$. Since $\mathbf{h}(\zeta,p)$ is generated by grouplike and skew primitive elements, $L^{\mathbf{h}(\zeta,p)} = L^{G(\mathbf{h}(\zeta,p))}= L^{\mathbb{Z}_n}$ by Theorem~\ref{inva}; let us denote this field by $F$. Then, $\mathbb{Z}_n$ acts faithfully on $L$. By Lemma~\ref{lem:Zn}, $L = \bigoplus_{i=0}^{n-1} L_{(i)}$, where $g \cdot r= \zeta^{-i} r$ for all $r \in L_{(i)}$.  We also have by Lemma~\ref{lem:Zn} that $L_{(0)} = F$ and $L = F[u]/(u^n-v)$ for $u \in L_{(1)}$ and $v$ a non-$n'$-th power in $F^{\times}$, for any $n'>1$ dividing $n$. 

Since 
$g \cdot (x_1 \cdot u) = \zeta x_1 \cdot (g \cdot u) = x_1 \cdot u,$
we get that $x_1 \cdot u \in F$ and we can renormalize to assume that $x_1 \cdot u =1$.  We also get that
$$x_1 \cdot u^d = (1+\zeta^{-1} + \cdots + \zeta^{-(d-1)})u^{d-1},$$
for all $d \geq 1$. Moreover, 
$g \cdot (x_2 \cdot u) = \zeta^p x_2 \cdot (g \cdot u) = \zeta^{p-1} (x_2 \cdot u),$
so we get that $x_2 \cdot u \in L_{(1-p)}$. Hence, $x_2 \cdot u = r_0 u^{1-p}$ for $r_0 \in F^{\times}$. 
Now, 
\[
\begin{array}{rl}
0~=~(x_1x_2-x_2x_1) \cdot u &= x_1 \cdot  (r_0 u^{1-p}) - x_2 \cdot 1\\
&=  r_0 (x_1 \cdot u^{1-p}) ~=~ r_0v^{-1}(x_1\cdot u^{n+1-p})\\
&= r_0 \left(1 + \zeta^{-1} + \cdots + \zeta^{-(n-p)}\right)u^{n-p}.
\end{array}
\]
So, $1 + \zeta^{-1} + \cdots + \zeta^{-(n-p)} = 0$, which implies that $p = 1$. 

For any $\mathbf{h}(\zeta,1)$-module field $L$, we have that the structure of $L$ is as described in Lemma~\ref{lem:Zn}.
\end{proof}


\subsection{The Hopf algebra $H_{81}$ is Galois-theoretical} \label{subsec:H81}
Let $\omega$ be a primitive cube root of unity. Let $H_{81}$ denote the 81-dimensional Hopf algebra from \cite[p.~1544]{Nichols}; see also \cite[Theorems~3.6 and~3.7]{AndSch:p4}. It is generated by a grouplike element $g$ and $(g,1)$-skew primitive elements $x, y$, subject to relations:
\[
\begin{array}{c}
g^3=1, \quad gx = \omega xg, \quad gy=\omega yg, \quad x^3=0, \quad y^3 = 0,\\ x^2y+xyx+yx^2 = 0, \quad y^2x  +yxy + xy^2 =0, \quad (xy-\omega yx)^3=0.
\end{array}
\]
Note that the relation $(xy-\omega yx)^3=0$ is accidentally omitted in \cite[p.~1544]{Nichols}.

\begin{proposition} \label{prop:H81}
The Hopf algebra $H_{81}$ is Galois-theoretical and the fields $L$ that admit an inner faithful $H_{81}$-action are precisely of the form $$L = F[u]/(u^3-v)$$ for $F  = L^{H_{81}}$, $u \in L_{(1)}$, and $v$ a non-cube element of $F^{\times}$. So, $L$ is a cyclic, degree~3 Galois extension of its subfield of invariants $F$ with Galois group $\mathbb{Z}_3$.
We have that
$g \cdot u = \omega^{-1}u$, $x\cdot u = w_1$, $y \cdot u = w_2 \in F$ for $w_1, w_2 \in F$ linearly independent over $\kk$. Here, $g \cdot r_0 = r_0, ~ x \cdot r_0  = y \cdot r_0= 0$ for all $r_0 \in F$. 
\end{proposition}

\begin{proof}
Applying Lemma~\ref{lem:Skryabin}, we have that $H_{81}$ is Galois-theoretical as it acts on $\kk[z]$ inner faithfully by
$$ 
g \cdot z = \omega^{-1} z, \quad x \cdot z = 1, \quad y \cdot z = z^3.
$$
(One also gets that $H_{81}$ acts inner faithfully on $\kk(z)$ by the same action.)
Indeed, it is clear that $g^3-1$, $gx - \omega xg$, and $gy-\omega yg$ act on $\kk[z]$ by zero. For the rest of the relations, note that any monomial 
in $x,y$ of degree $\ge 3$ acts by zero in $\kk[z]$.
To determine $H_{81}$-module fields $L$, 
first observe that $G(H_{81}) = \mathbb{Z}_3$. By an argument similar to that in Section~\ref{subsec:E(n)}, the result holds.
\end{proof}

We also have the following generalization of Proposition \ref{prop:H81}.

\begin{proposition} \label{prop:H81skewprod} Retain the notation above.
Let $G$ be a finite subgroup of $GL_2(\kk)$.  Then one can form 
the semi-direct product $\kk G \ltimes H_{81}$  where $GL_2(\kk)$ acts on $H_{81}$ by linear transformations of the skew primitive elements $x$ and $y$.
Moreover, $\kk G\ltimes H_{81}$ is Galois theoretical. 
\end{proposition}

\begin{proof}
For the first statement, one can check  that the ideal of relations of $H_{81}$ is stable under the action of $GL_2(\kk)$. To get the last statement, adapt the proof of Proposition~\ref{prop:E(n)skewprod}.
\end{proof}


\subsection{The Hopf algebra $u_q(\mf{sl}_2)$ is Galois-theoretical, but gr($u_q(\mf{sl}_2)$) is not} \label{subsec:uqsl2} 

Let $m \geq 2$ and let $q$ be a root of unity in $\kk$ with ord($q^2$)=$m$. Consider the $m^3$-dimensional Hopf algebra $H_{\lambda}$,  generated by a grouplike element $k$, a $(k,1)$-skew primitive element $e$, and a $(1,k^{-1})$-skew primitive element $f$. Let $H_{\lambda}$ have relations:
$$ef-fe = \lambda \frac{k-k^{-1}}{q-q^{-1}}, \quad ke= q^2 ek, 
\quad kf = q^{-2}fk, \quad e^{m} = f^{m}=0, \quad k^{m} =1.$$

Note that if $\lambda \neq 0$, then $H_{\lambda} \cong u_q(\mf{sl}_2)$, and without loss of generality we can take $\lambda =1$ in this case. Otherwise, $H_{\lambda = 0}$ is isomorphic to the associated graded Hopf algebra gr($u_q(\mf{sl}_2)$). Part (b) of the result below recovers \cite[Corollary~3.7]{MontgomerySchneider}.

\begin{proposition} \label{prop:uqsl2} We have the following statements.
\begin{enumerate}
\item The associated graded Hopf algebra gr($u_q(\mf{sl}_2)$) is not Galois-theoretical. 
\item The Hopf algebra $u_q(\mf{sl}_2)$ is Galois-theoretical and the fields $L$ that admit an inner faithful $u_q(\mf{sl}_2)$-action are precisely of the form $L = F[u]/(u^{m}-v)$ for $F  = L^{u_q(\mf{sl}_2)}$, $u \in L_{(1)}$, and $v$ is a non-$m'$-th root in $F^{\times}$, for any $m'>1$ dividing $m$. In other words, $L$ is a cyclic degree $m$ Galois extension of its subfield of invariants $F$ with Galois group $\mathbb{Z}_m$.
Moreover, we have that
\[
\begin{array}{lllll}
e \cdot u = 1,  && f \cdot u =-qu^2, && k \cdot u =q^{-2}u,
\end{array}
\]
and $e \cdot r_0= f \cdot r_0 = 0$, ~$k \cdot r_0 = r_0$
for all $r_0 \in F$.
\end{enumerate}
\end{proposition}

\begin{proof}  
(a) Suppose that $H_{\lambda}$ is Galois-theoretical with module field $L$; we will show that $\lambda \neq 0$. The subalgebra generated by $\{k,e\}$, which is isomorphic  to the Taft algebra $T(m)$, acts inner faithfully on $L$ by Proposition~\ref{prop:GTprelim}(c). By Lemma~\ref{lem:Zn} and Proposition~\ref{prop:T(n)}, $L = \bigoplus_{i=0}^{m-1} L_{(i)}= F[u]/(u^{m}-v)$ where $L_{(i)} = \{r \in L ~|~ k \cdot r = q^{-2i} r\}$, so $L_{(0)}=L^{T(m)}=:F$ and $u \in L_{(1)}$. So for $u \in L_{(1)}$ and  $r_0 \in F$, we have that
$$k \cdot u = q^{-2}u, \quad e \cdot u =1, \quad k \cdot r_0 = r_0, \quad e \cdot r_0 = 0.$$
Since $k \cdot (f \cdot u) = q^{-2} f \cdot (k \cdot u) = q^{-4} (f \cdot u)$, we get that $f \cdot u = r_0 u^2$ for some $r_0 \in F^{\times}$.

Now we use the relation $ef-fe = \lambda \frac{k-k^{-1}}{q-q^{-1}}$ to verify part (a). On the one hand, we have that
\[
\begin{array}{rl}
(ef-fe) \cdot u &= ~e \cdot (r_0 u^2) - f \cdot 1
~=~ r_0 (e \cdot u^2) \\ 
&=~r_0 \left((k \cdot u) (e \cdot u) + (e \cdot u)u \right) ~=~ r_0(q^{-2}+1) u.
\end{array}
\]
On the other hand, we get that
$$\left(\lambda \frac{k-k^{-1}}{q-q^{-1}}\right) \cdot u ~=~ \frac{\lambda}{q-q^{-1}} (q^{-2} - q^2) u. $$
Thus, 
$$r_0 ~=~ \lambda \frac{(q^{-2} - q^2)}{(q-q^{-1})(q^{-2}+1)}  ~=~ - \lambda q.$$
Since $r_0 \in F^{\times}$, we must have that $\lambda \neq 0$, as required.

(b) Here, we show that $u_q(\mf{sl}_2)$ is Galois-theoretical, then we use the work in part (a) to determine the structure of its module fields. 
First,  $u_q(\mf{sl}_2)$ acts  on the polynomial ring $\kk[z]$ and the field $\kk(z)$ by
\[{
\begin{array}{lll}
e \cdot z =1, ~~ &f \cdot z = -qz^2, ~~ &k \cdot z = q^{-2}z.
\end{array}
}\] 
The action is inner faithful as the skew primitive elements do not act by zero; see \cite[Corollary~5.4.7]{Montgomery}.
Hence, $u_q(\mf{sl}_2)$ is Galois-theoretical. Now for any $u_q(\mf{sl}_2)$-module field $L$, we have that 
$L^{u_q(\mf{sl}_2)} = L^{G(u_q(\mf{sl}_2))} = L^{\mathbb{Z}_{m}} =: F$ where  $\mathbb{Z}_{m}$ acts faithfully on $L$  by Theorem~\ref{inva}. By Lemma~\ref{lem:Zn}, the structure of $L$ is as claimed and part (b) holds.
\end{proof}

We also have a slight reformulation of Proposition~\ref{prop:uqsl2}(b), which will be used in the sequel of this article. Let $q$ be a primitive $m$-th root of unity.
Let $K_q$ be the $m^3$-dimensional Hopf algebra generated by the grouplike element $g$ and $(g,1)$-skew primitive elements $x$ and $y$, subject to relations: $$g^m=1, \quad x^m=y^m=0, \quad gx=qxg, \quad gy=q^{-1}yg, \quad yx-qxy=1-g^2.$$

\begin{proposition} \label{prop:Kq}
The Hopf algebra $K_q$ is Galois-theoretical.
\end{proposition}

\begin{proof}
We see that $K_{q^2}$ is isomorphic to $u_q(\mf{sl}_2)$, where we identify $g$, $x$, $y$ with $k$, $e$, $(q-q^{-1})kf$, respectively. 
\end{proof}

\subsection{Galois-theoretical twists of $u_q(\mf{gl}_n)$ and of $u_q(\mf{sl}_n)$} \label{subsec:uqslntwist}

In this section, let $q \in \kk$ be a root of unity of odd order $m \geq 3$ as in \eqref{eq:order}, and let $n\geq 2$. Recall the definition of the Kac-De Concini quantum group ${\mathcal{U}}_q(\mathfrak{sl}_n)$
 and the small quantum group
$u_q(\mathfrak{sl}_n)$ from Section \ref{rootof1}. 
In this subsection, we need extensions of these quantum groups associated to ${\mathfrak{gl}}_n$. 

To define these extensions, we first define commuting automorphisms $g_i$ of 
${\mathcal{U}}_q(\mathfrak{sl}_n)$, for $i=1,...,n$, by the formulas 
$$
g_i(k_j)=k_j,\quad g_i(e_j)=q^{\delta_{ij}-\delta_{i,j+1}}e_j,\quad
g_i(f_j)=q^{-\delta_{ij}+\delta_{i,j+1}}f_j.
$$
It is easy to see we get that $g_ig_{i+1}^{-1}$ coincides 
with the inner automorphism defined by the grouplike element $k_i$ for each $i = 1, \dots, n-1$.  
Moreover, the automorphisms $g_i$ clearly descend to the quotient Hopf algebra 
$u_q(\mathfrak{sl}_n)$, where they satisfy the relations $g_i^m=1$. This prompts the following definition.

\begin{definition} The Hopf algebra 
${\mathcal{U}}_q(\mathfrak{gl}_n)$ is
 the smash product of ${\mathcal{U}}_q(\mathfrak{sl}_n)$
with the group $\Bbb Z^n$ generated by the $g_i$, 
modulo the relations $g_ig_{i+1}^{-1}=k_i$. 

The finite dimensional Hopf algebra $u_q(\mathfrak{gl}_n)$
is the smash product of $u_q(\mathfrak{sl}_n)$
with the group $(\Bbb Z/m\Bbb Z)^n$ generated by the $g_i$, 
modulo the relations $g_ig_{i+1}^{-1}=k_i$. 
\end{definition}

More explicitly, $u_q(\mathfrak{gl}_n)$ is
the Hopf algebra generated by grouplike elements $g_i$ for $i=1,...,n$,
$(k_j,1)$-skew primitive elements $e_j$, and $(1,k_j^{-1})$-skew
primitive elements $f_j$, for 
$k_j:=g_jg_{j+1}^{-1}$, with $j =1, \dots, n-1$, subject to relations:
\[ 
\begin{array}{ll} 
g_ie_jg_i^{-1} = q^{\delta_{ij}-\delta_{i,j+1}}e_j, 
\quad \quad g_if_jg_i^{-1} = q^{-\delta_{ij}+\delta_{i,j+1}} f_j, &\\
e_ie_j= e_je_i, \hspace{.95in} f_if_j=f_jf_i, &(|i-j| \geq 2)\\
e_i^2e_j - (q+q^{-1})e_ie_je_i + e_je_i^2 =0,
&  (|i-j|=1)\\ 
f_i^2f_j - (q+q^{-1})f_if_jf_i + f_jf_i^2 =0,
& (|i-j|=1) 
\end{array} 
\]
\vspace{-.4in}

\[
\hspace{0in}
\begin{array}{ll}
g_ig_j = g_jg_i, \hspace{.9in} e_if_j - f_je_i = \delta_{ij}
\frac{k_i - k_i^{-1}}{q-q^{-1}}, \hspace{.45in}  & \\
g_i^m=1, \hspace{1.1in} (e_{\alpha}^S)^m = (f_{\alpha}^S)^m =0, &\alpha>0
\end{array}
\]
\smallskip

\noindent where $e_{\alpha}^S$ and $f_{\alpha}^S$ are the quantum root elements 
attached to a reduced decomposition $S$
of the maximal element $w_0$ of the symmetric group, 
as in Section \ref{rootof1}. 
It is easy to see that  $u_q(\mathfrak{gl}_n)$ 
has dimension $m^{n^2}$. 

In our first result of this section, we show that $u_q(\mathfrak{gl}_2)$ is Galois-theoretical.

\begin{proposition} \label{prop:uq(gl2)} The Hopf algebra $u_q(\mf{gl}_2)$ is Galois-theoretical and the fields $L$ that admit an inner faithful $u_q(\mf{gl}_2)$-action are precisely of the form $$L = F[u,u']/(u^{m}-v, u'^m - v')$$ for $F  = L^{u_q(\mf{gl}_2)}$, $u \in L_{(1)}$, for some $v, v' \in F^{\times}$ so that $L$ is a field. In other words, $L$ is a Galois extension of its subfield of invariants $F$ with Galois group $\mathbb{Z}_m \times \mathbb{Z}_m$.
\end{proposition}

\begin{proof}
We have that $u_q(\mf{gl}_2)$ acts  on the field $\kk(z)$ by extending the action of $u_q(\mf{sl}_2)$ on $\kk(z)$ from Proposition~\ref{prop:uqsl2}(b) as follows:
\[
\begin{array}{llll}
\smallskip
g_1 \cdot z= q^{-1} z, \quad & g_2 \cdot z = q z, \quad & e \cdot z = 1, \quad & f \cdot z = -q z^2.
\end{array}
\]
The action is inner faithful as the skew primitive elements do not act by zero; see \cite[Corollary~5.4.7]{Montgomery}.
Hence,  $u_q(\mf{gl}_2)$ is Galois-theoretical.
Also  by Theorem~\ref{inva}, $L^{u_q(\mf{gl}_2)}=L^{\Bbb{Z}_m\times \Bbb{Z}_m}$,
which implies the second statement.
\end{proof}

To study the Galois-theoretical property of twists of $u_q(\mathfrak{gl}_n)$ and of $u_q(\mathfrak{sl}_n)$, consider the quantum polynomial algebra 
$$A_q = \kk\langle z_1, \dots, z_n \rangle/(z_iz_j - q
z_j z_i)_{i <j}.$$ By \cite[Theorem~4.1]{Hu},
we have that $A_q$ is a left 
$\mathcal{U}_q({\mathfrak{gl}}_n)$-module algebra with the following action:
\[
\begin{array}{llll} 
e_i \cdot z_{i+1} = z_i, ~  \quad &f_i \cdot z_i =z_{i+1}, ~\quad
& g_i \cdot z_j = q^{\delta_{ij}}z_j, \\
e_i \cdot z_j = 0, ~ \quad &f_i \cdot z_{j'} =0,\\
\end{array} 
\] 
 for $j \neq i+1$, $j' \neq i$. Thus, we have the following result.

\begin{lemma}\label{descend}
The action of $\mathcal{U}_q(\mathfrak{gl}_n)$ on $A_q$ above descends to an inner faithful action of 
$u_q({\mathfrak{gl}}_n)$ on $A_q$.
\end{lemma}

\begin{proof}
We have that $g_i^m -1$, $(e_{\alpha}^S)^m$, $(f_{\alpha}^S)^m$ generate a Hopf ideal of $\mathcal{U}_q(\mathfrak{gl}_n)$. So, to check that this Hopf ideal acts by zero on $A_q$, it suffices to check that it acts by zero on the generators of $A_q$, in this case, $z_i$. It is obvious that $g_i^m$ acts as the identity on $A_q$. Moreover, $(e_{\alpha}^S)^m$ (resp., $(f_{\alpha}^S)^m$) act by zero on the generators $z_j$,  as $(e_{\alpha}^S)^m$ (resp., $(f_{\alpha}^S)^m$) contains more than one copy of some $e_i$ (resp. some $f_i$). 

Since $u_q(\mathfrak{gl}_n)$ is of finite Cartan type, and thus generated  by the degree one part of its coradical filtration, the only skew primitive elements of $u_q(\mf{gl}_n)$ modulo the trivial ones, up to multiplication by grouplike elements and up to scaling, are $e_i$ and $f_i$.
\footnote{Let $H=gr(u_q(\mf{gl}_n))$. Then, it is known that $H^*$ is a Hopf subalgebra in $u_q^{\geq 0}(\mf{gl}_n)\otimes u_q^{\leq 0}(\mf{gl}_n)$, 
so it is generated in degree 1, i.e. by the grouplike elements and $e_i,f_i$. (This is also a special case of Theorem~\ref{thm:ASconjAbel}, as 
$H^*$ is pointed with an abelian group of grouplike elements.) This implies that any 
homogeneous skew-primitive element $x$ of $H$ of degree $\ge 2$ is zero. Indeed, $\langle x,ab \rangle= \langle \Delta(x),a\otimes b \rangle=0$ 
if $a,b\in H^*$ are of positive degree, but any element of degree $\ge 2$ in $H^*$ is a linear combination of elements of the form $ab$ with 
$\deg(a),\deg(b)\ge 1$. Thus, any skew-primitive element in $H$ modulo trivial ones is a product of a grouplike element with $e_i$ or $f_i$ up to scaling. 
Hence, the same is true for $u_q(\mf{gl}_n)$. }
Hence, the action of  $u_q(\mathfrak{gl}_n)$ on $A_q$ is inner faithful since any nonzero Hopf ideal of $u_q(\mathfrak{gl}_n)$ has a nonzero intersection with the span of skew primitive elements of $u_q(\mathfrak{gl}_n)$ \cite[Corollary~5.4.7]{Montgomery}.
\end{proof}
  
Recall the discussion in Section~\ref{subsec:twist}. Let $G=
(\mathbb{Z}/m\mathbb{Z})^n$  be the Cartan subgroup of $u_q(\mathfrak{gl}_n)$ and let 
$\chi_i \in \widehat{G}$ be defined as $\chi_i(p_1, \dots, p_n) =
q^{p_i}$. 
Let $J^+$ and $J^-$ be Drinfeld twists of $\kk G$ so that 
$$\sigma_{J^+}(\chi_i,\chi_j) = 
\begin{cases} 
q &\text{ for }i>j,\\ 1 &\text{ for
}i\le j 
\end{cases} \text{ \quad and \quad }  
\sigma_{J^-}(\chi_i,\chi_j) =  
\begin{cases} 
q &\text{ for } i < j,\\ 1 &\text{ for }i\ge j.
\end{cases}$$
Note that the twist $(J^\pm)^{-1}$ is gauge equivalent to $J^\mp$.

Let us identify $G$ with $G(H)$ via
$(p_1,\dots,p_n)\mapsto g_1^{p_1}\dots g_n^{p_n}$. 
Then, we have the following result. 

\begin{proposition} \label{prop:uqglntwist} 
The twists $u_q(\mf{gl}_n)^{J^+}$ and $u_q(\mf{gl}_n)^{J^-}$ are Galois-theoretical.
\end{proposition} 

\begin{proof} 
Since $A_q$ is an inner faithful left $u_q(\mf{gl}_n)$-module algebra, $( A_q)_{J^+}$ is an inner faithful left $u_q(\mf{gl}_n)^{J^+}$-module 
algebra. Now by Lemma~\ref{lem:Skryabin}, it suffices to show
that $(A_q)_{J^+}$ is a commutative domain. 
By Theorem~\ref{thm:GKM}, we get that $\kk[z_1,
\dots,z_n]_{(J^+)^{-1}} \cong A_q$. 
Thus, $\kk[z_1, \dots,z_n] = (\kk[z_1,
\dots,z_n]_{(J^+)^{-1}})_{J^+} = (A_q)_{J^+}$.

By using the map $\Phi$ that relabels indices by $i \mapsto
n+1-i$, we get that $A_{q^{-1}} = \Phi(A_q)$ is a left
$u_q(\mf{gl}_n) = 
\Phi(u_q(\mf{gl}_n))$-module algebra. Following the argument above, we get that $u_q(\mf{gl}_n)^{J^-}$ is also Galois-theoretical.
\end{proof}

Proposition \ref{prop:uqglntwist} allows us to show that some
quotients of $u_q(\mf{gl}_n)^{J^\pm}$ are also Galois-theoretical. 
Namely, let $C$ be the subgroup of central grouplike elements in
$G=G(H)$. It is clear that an element $g=g_1^{p_1} \dots g_n^{p_n}$
is central if and only if $p_i=p_{i+1}$ for all $i$, that is to say,
$g:=(g_1 \dots g_n)^t$ for some integer $t \geq 1$. So the group $C$ is isomorphic to
$\Bbb Z/m\Bbb Z$ and is generated by $c:=g_1 \dots g_n$. Now, 
consider the Hopf algebra
$u_q(\mf{gl}_n)^{[s]}:=u_q(\mf{gl}_n)/(c^s-1)$.

\begin{proposition}
We have that $(u_q(\mf{gl}_n)^{[s]})^{J^\pm}$ is Galois-theoretical.
\end{proposition}

\begin{proof}
Let $L$ be the $u_q(\mf{gl}_n)^{J^\pm}$-module field obtained from
Proposition \ref{prop:uqglntwist}. 
Let $Q=\lbrace{z\in L ~|~ c^s \cdot z=z\rbrace}$. Then, $Q$ is
an $(u_q(\mf{gl}_n)^{[s]})^{J^\pm}$-module field, and it is easy to check
directly that the action of $(u_q(\mf{gl}_n)^{[s]})^{J^\pm}$ on this field is
inner faithful.  
\end{proof} 

The algebra $u_q(\mf{sl}_n)$ has a subgroup $\mathbb{Z}/n\mathbb{Z}$ consisting of central grouplike elements, and its intersection with the group $\mathbb{Z}/m\mathbb{Z}$ generated by $c$ is $\mathbb{Z}/$gcd($m,n$)$\mathbb{Z}$. So the intersection of these subgroups is trivial if and only if gcd($m$,$n$)$=1$. Thus, $u_q(\mf{gl}_n)/(c-1) = u_q(\mf{sl}_n)$ for $m$ and $n$ relatively prime, and we have the following result.

\begin{corollary} \label{cor:uqslntwist}
If $m$ and $n$ are relatively prime, then we have that $u_q({\mathfrak{sl}}_n)^{J^\pm}$ are 
Galois-theoretical (where we abuse notation and denote by 
$J^\pm$ the images of the twists $J^\pm$ in the quotient).  \qed
\end{corollary}

 In general, we get that $(u_q({\mathfrak{sl}}_n)/(c^s-1))^{J^\pm}$ is 
Galois-theoretical for $s=m/gcd(m,n)$.

\begin{remark} \label{rem:fullII} 
We will show in Part II of this work that, in contrast with the above result, untwisted 
$u_q({\mathfrak{sl}}_n)$ is not Galois-theoretical for $n\ge 3$.
We will also show that  the twists $J^+$ and $J^-$ are the only twists $J$ coming from the Cartan subgroup that make 
$u_q({\mathfrak{sl}}_n)^J$  Galois-theoretical. For $u_q({\mathfrak{gl}}_n)$, the situation is similar. 
\end{remark}

\subsection{A modification of $u_{q^{1/2}}(\mf{gl}_2)$ that is Galois-theoretical} \label{subsec:moduqgl2} 
Let $m \geq 2$ and let $q$ be a primitive $m$-th root of unity in $\kk$. We consider a modification $u'_q(\mf{gl}_2)$ of the Hopf algebra $u_{q^{1/2}}(\mf{gl}_2)$ that is of finite Cartan type. We will see that this is a special case of the Galois-theoretical Hopf algebra 
$u_{q^{1/2}}(\mf{gl}_2)^{J^+}$ considered above. 
The computations below follow similarly to those in previous sections, so some details are omitted.

\begin{definition} The $m^4$-dimensional Hopf algebra $u'_q(\mf{gl}_2)$ is generated by grouplike elements $\gamma_1, \gamma_2$, a $(\gamma_1,1)$-skew primitive element $x_1$, and a $(\gamma_2, 1)$-skew primitive element $x_2$, subject to relations:
\[
\begin{array}{c}
\gamma_1^m = \gamma_2^m =1, \quad \gamma_1\gamma_2 = \gamma_2 \gamma_1, \quad x_1^m = x_2^m =0, 
\quad x_2x_1 - q x_1 x_2= 1- \gamma_1 \gamma_2,\\
\gamma_1x_1 = q x_1 \gamma_1, \quad \gamma_1x_2 = q^{-1} x_2 \gamma_1, \quad \gamma_2x_1 = q x_1 \gamma_2, \quad \gamma_2x_2 = q^{-1} x_2 \gamma_2.
\end{array}
\]
\end{definition}

\noindent We have the following two results.

\begin{proposition}
 Let $q$ be a primitive $m$-th root of unity in $\kk$ as in \eqref{eq:order}. 
Then, we have an isomorphism of Hopf algebras 
$\phi: u_{q^2}'({\mathfrak{gl}}_2)\to u_{q}({\mathfrak{gl}}_2)^{(J^+)^{-1}}$ 
given by the formulas 
$$
\phi(\gamma_1)=g_1^{2},\quad \phi(\gamma_2)=g_2^{-2},\quad \phi(x_1)=eg_1,\quad \phi(x_2)=(q-q^{-1})g_2^{-1}f,
$$
where the twist $J^+$ is defined in Section~\ref{subsec:uqslntwist} for $n=2$. 
\end{proposition}

\begin{proof}
One can check this by direct computation. Here, $G = ({\mathbb Z}/m{\mathbb Z})^{2}$ is the group of grouplike elements of $u_q(\mathfrak{gl}_2)$ and we have that $J^+ = \sum_{\chi, \psi \in \widehat{G}} \sigma_{J^+}(\chi, \psi) {\mathbf 1}_{\chi} \otimes {\mathbf 1}_{\psi},$ where ${\mathbf 1}_{\chi} e =e {\mathbf 1}_{\chi - \chi_1 + \chi_2}$ and ${\mathbf 1}_{\chi} f =f {\mathbf 1}_{\chi + \chi_1 - \chi_2}$.
\end{proof}

\begin{proposition} \label{prop:moduqgl2} 
The Hopf algebra $u'_q(\mf{gl}_2)$ is Galois-theoretical, and the fields $L$ that admit an inner faithful $u'_{q}(\mf{gl}_2)$-action are of the form $$L ~=~ F[u,u']/(u^m-v, u'^m - v')$$ for $F  = L^{u'_{q}(\mf{gl}_2)}$, for some $v, v' \in F^{\times}$ so that $L$ is a field. In other words, $L$ is a Galois extension of its subfield of invariants $L^{u'_{q}(\mf{gl}_2)}$ with Galois group $\mathbb{Z}_m \times \mathbb{Z}_m$.
\end{proposition}

\begin{proof}
We have that $u_q'(\mf{gl}_2)$ acts inner faithfully on the field $\kk(z_1,z_2)$ by
\[ 
 \hspace{-.07in}\begin{array}{llll}
\smallskip
\gamma_1 \cdot z_1 = q z_1, \quad \quad  & \gamma_1 \cdot z_2 = z_2, \quad \quad &\gamma_2 \cdot z_1 = z_1, \quad \quad & \gamma_2 \cdot z_2 = q z_2,\\
\smallskip
x_1 \cdot z_1= (1-q) z_1^2z_2, &x_1 \cdot z_2 =0,
&x_2 \cdot z_1 =0,& x_2 \cdot z_2 = \frac{1}{z_1}.\\
\end{array}
\]
Hence,  $u_q'(\mf{gl}_2)$ is Galois-theoretical.
Also  by Theorem~\ref{inva}, $L^{u_q'(\mf{gl}_2)}=L^{\Bbb{Z}_m\times \Bbb{Z}_m}$,
which implies the second statement.
\end{proof}


\subsection{Galois-theoretical twists of $u_q^{\geq 0}(\mf{g})$} \label{subsec:uqgtwist}

Keep the notation of Section \ref{rootof1}.
First, let $q$ be a variable (i.e., we work over $\kk[q,q^{-1}]$). 
Fix an orientation of edges on the Dynkin diagram of ${\mathfrak{g}}$,
and denote the corresponding oriented diagram by $Q$.  

To examine the Galois-theoretical property of twists of $u_q^{\geq 0}(\mathfrak{g})$, we consider the quantum polynomial algebra 
$$
A_{q,Q}: = \kk\langle z_1, \dots, z_r \rangle/(z_iz_j - q^{\pm
  d_ia_{ij}} z_j z_i)_{i <j},
$$
$i=1,...,r$, where $\{z_i\}$ correspond to the vertices of the Dynkin
diagram $Q$. Here, 
the power of $q$ is $d_i a_{ij}$ if the edge $i-j$ is oriented
as $i \rightarrow j$, and is $-d_i a_{ij}$ otherwise.

We have the following well known proposition, which can be proved
directly.

\begin{proposition} \cite{IoharaMalikov} \cite[Proposition~3.1]{Etingof:Whittaker}
The algebra $A_{q,Q}$ is a quotient of  the subalgebra
${\mathcal{U}}_q^+(\mf{g})$ of ${\mathcal{U}}_q(\mf{g})$ 
generated by the $\{e_i\}$, with the quotient map sending $e_i$ 
to $z_i$. Namely, $A_{q,Q}\cong {\mathcal{U}}_q^+(\mf{g})/I_Q$, where $I_Q =  \langle e_i
e_j -q^{\pm d_i a_{ij}} e_j e_i \rangle_{i<j}$. \qed
\end{proposition}

The next proposition claims that the adjoint action of 
${\mathcal{U}}_q^{\geq 0}(\mf{g})$ on ${\mathcal{U}}_q^+(\mf{g})$ descends to 
an action on $A_{q,Q}$. 

\begin{proposition} \label{prop:UactonAqQ}
We have that $A_{q,Q}$ is a left ${\mathcal{U}}_q^{\geq 0}(\mf{g})$-module
algebra, where the action is induced by the (left) adjoint action
of ${\mathcal{U}}^{\ge 0}_q(\mf{g})$ on itself.  
In other words, $h \cdot a = \sum h_1 a S(h_2)$ for $h,a \in {\mathcal{U}}^{\ge 0}_q(\mf{g})$, so
\[
k_i \cdot e_j= q^{d_ia_{ij}} e_j\quad \text{and} \quad
e_i \cdot e_j = -q^{d_ia_{ij}}e_j e_i +e_i e_j.
\]
\end{proposition}

\begin{proof} 
To verify the claim, it suffices to show that the ideal $I_Q$ is ${\mathcal{U}}_q^+(\mf{g})$-stable under the adjoint action.  Indeed, it is clear that the action of $k_{\ell}$ stabilizes $I_Q$. 
Note that
$$e_{\ell} \cdot u = e_{\ell} u - k_{\ell} u k_{\ell}^{-1}e_{\ell},$$
for any $u \in {\mathcal{U}}_q^+(\mf{g})$. So, any two sided ideal of ${\mathcal{U}}_q^+(\mf{g})$ stable under the adjoint actions of $\{k_{\ell}\}$ is also stable under the action of $\{e_{\ell}\}$, and we are done.
\end{proof} 

Now let us specialize $q$ to a root of unity of order $m$ as in \eqref{eq:order} of Section \ref{rootof1}. Moreover, let $C$ be the subgroup of central grouplike elements of $u_q^{\ge 0}({\mathfrak{g}})$; it consists of elements $\prod_i k_i^{\ell_i}$, such that $\sum_i \ell_id_ia_{ij}$ is divisible by $m$ for all $j$. 
Then we have the following proposition.

\begin{proposition}\label{descend1}
Assume \eqref{eq:order}. The action of ${\mathcal U}^{\ge 0}_q({\mathfrak{g}})$ on $A_{q,Q}$ descends to an action of 
$u_q^{\ge 0}({\mathfrak{g}})/(c-1)_{c \in C}$. Moreover, this action is inner faithful. 
\end{proposition}

\begin{proof}
Let $J$ be the kernel of the projection ${\mathcal{U}}_q^{\ge 0}({\mathfrak{g}})\to u_q^{\ge 0}({\mathfrak{g}})$, which is a Hopf ideal. 
We need to show that $J$ acts by zero in $A_{q,Q}$, i.e. that  $k_i^m-1$ and $(e_\alpha^S)^m$ act by zero. 
Let $Z_0^{\geq 0}\subset \mathcal{U}_q^{\geq 0}(\mathfrak g)$  be the subalgebra generated by these elements. 
 By \cite[Proposition~5.6(d)]{DKP} and \cite[Corollary~3.1]{deConciniKac}, we have that 
$Z_0^{\geq 0}$ is a Hopf subalgebra of $\mathcal{U}_q^{\geq 0}(\mathfrak g)$ generated by central elements. Hence, if $h\in Z_0^{\geq 0}$, then 
$h\cdot a= \sum h_1aS(h_2)= \sum ah_1S(h_2)=\varepsilon(h)a$, as desired. 
The inner faithfulness is clear, as the kernel for the action of the grouplike elements is exactly $C$, and 
all skew primitive elements $e_i$ of $u_q^{\geq 0}(\mathfrak{g})$ act nontrivially by the definition of the $\mathcal{U}_q^{\geq 0}(\mathfrak{g})$-action on $A_{q,Q}$ from Proposition~\ref{prop:UactonAqQ}.
\end{proof}

Recall the discussion in Section~\ref{subsec:twist}. Let $G= (\mathbb{Z}/m\mathbb{Z})^{ r}=G(u_q^{\geq 0}({\mathfrak{g}}))$, and let $\alpha_i \in \widehat{G}$ be the simple root characters 
defined by $\alpha_i(k_j)=q^{d_ia_{ij}}$. Assume that 
\begin{equation} \label{eq:detaij}
\text{$m$=ord($q$) of (\ref{eq:order}) is relatively prime to  det$(a_{ij})$, and to $3$ in type $G_2$.}
\end{equation}

In this case, $C=\{1\}$, and $\alpha_i$ are independent generators of $\widehat{G}$.
Thus, there is a unique, up to gauge transformations, Drinfeld twist $J_Q$ of $\kk G$ such that 
\begin{equation} \label{bJ}
b_{J_Q}(\alpha_i,\alpha_j) = 
\begin{cases}
q^{d_ia_{ij}} &\text{ for }i\rightarrow j \text{ in $Q$},\\ q^{-d_ia_{ij}} &\text{ for }i\leftarrow j \text{ in $Q$},\\ 1 &\text{ for } i\text{ not connected to } j \text{ in $Q$}.
\end{cases}
\end{equation}
To see this, recall Proposition~\ref{prop:JtobJ}.
For example, one may take the twist $J_Q$ defined by 
\begin{equation}\label{sigmaj}
\sigma_{J_Q}(\alpha_j,\alpha_i) = 
\begin{cases}
q^{d_ia_{ij}} &\text{ for }i\rightarrow j \text{ in $Q$},\\ 1 &\text{ for }i\leftarrow j \text{ in $Q$},\\ 1 &\text{ for } i\text{ not connected to } j \text{ in $Q$}.
\end{cases}
\end{equation}
So we have $2^{r-1}$ such twists, up to gauge transformations.
Namely, they are parametrized by orientations of the Dynkin diagram, 
which has $r-1$ edges, where $r=\text{rank}(\mathfrak{g})$. Then, we have the following result.

\begin{proposition} \label{prop:uqgpostwist} 
Assume \eqref{eq:detaij}. Then, the twists $u_q^{\geq 0}(\mf{g})^{J}$  are Galois-theoretical for each $J=J_Q$ as in \eqref{bJ}.
\end{proposition}

\begin{proof}
By Proposition \ref{descend1} and under the assumption of \eqref{eq:detaij}, we have that $A_{q,Q}$ is a left $u_q^{\geq 0}(\mf{g})$-module algebra. Hence,  $(A_{q,Q})_J$ is a left $u_q^{\geq 0}(\mf{g})^J$-module algebra. Now by Lemma~\ref{lem:Skryabin}, it suffices to show that $(A_{q,Q})_J$ is a commutative domain. By Theorem~\ref{thm:GKM}, we get that $\kk[z_1, \dots,z_n]_{J^{-1}} = A_{q,Q}$. Thus, $\kk[z_1, \dots,z_n] = (\kk[z_1, \dots,z_n]_{J^{-1}})_J = (A_{q,Q})_J$.
\end{proof}

\begin{remark}\label{rem:posII}
We will show in Part II of this work that the twists $J_Q$ above are the only ones coming from the Cartan subgroup  of $u_q^{\geq 0}(\mathfrak{g})$ so that $u_q^{\geq 0}(\mathfrak{g})^{J_Q}$ is Galois-theoretical. In particular, 
the Hopf algebra $u_q^{\ge 0}({\mathfrak{g}})$ is not Galois-theoretical, unless ${\mathfrak{g}}={\mathfrak{sl}}_2$. 
We will also see in Part II that the full small quantum group $u_q({\mathfrak{g}})$ does not 
become Galois-theoretical under twists coming from the Cartan subgroup, unless ${\mathfrak{g}}={\mathfrak{sl}}_n$. 
\end{remark}

If we are not in the setting of \eqref{eq:detaij}, that is to say, if $m$ is not relatively prime to the determinant of the Cartan matrix (or to $3$ for type $G_2$), then the twists $J_Q$ as above do not exist in general. 
Indeed, consider the case of type $A_{n-1}$. Then $\det(a_{ij})=n$ and $d_i=1$. Let $\omega_i\in \widehat{G}$ be such that $\omega_i(k_j)=q^{\delta_{ij}}$,
so that $\alpha_i=\prod_k\omega_k^{a_{ik}}$. Let $b_J(\omega_k,\alpha_j)=c_{kj}$. 
Then from \eqref{bJ}, we get that:
\[
\begin{array}{c}
\displaystyle \prod_k c_{kj}^{a_{ik}}=q^{a_{ij}},\text{ if } i\rightarrow j, \quad \quad
\displaystyle \prod_k c_{kj}^{a_{ik}}=q^{-a_{ij}},\text{ if } i \leftarrow j,\\
\displaystyle \prod_k c_{kj}^{a_{ik}}=1,\text{ if } i\text{ not connected to } j.
\end{array}
\]
Recall that $c_{kj}$ are $m$-th roots of $1$; so let $c_{ij}=q^{b_{ij}}$, $b_{ij}\in \Bbb Z/m\Bbb Z$. 
 Then, we get that:
 \[
\begin{array}{c}
\displaystyle \sum_k a_{ik}b_{kj}=a_{ij}, \text{ if } i\rightarrow j,  \quad \quad
\displaystyle \sum_k a_{ik}b_{kj}=-a_{ij}, \text{ if } i\leftarrow j, \\
\displaystyle \sum_k a_{ik}b_{kj}=0,
\text{ if } i\text{ not connected to } j
\end{array}
\]
in $\Bbb Z/m\Bbb Z$.  
 
Assume that $gcd(m,n)=d$. The equations above yield 
\begin{equation} \label{eq:1}
\sum_{k=1}^{n-1} a_{ik}b_{kj} = s_{ij} a_{ij} \mod m,
\end{equation}
 where $s_{ij}$ equals 1 if $i \rightarrow j$, equals $-1$ if $i \leftarrow j$, and equals $0$ if $i$ is not connected to $j$. 
We also have that $\sum_{i=1}^{n-1} ia_{ik}=0 \mod n$. Hence, 
$
\sum_{i=1}^{n-1} \frac{im}{d} a_{ik} =0 \mod m.$
 (Indeed, if $\ell$ is divisible by $n$, then $m\ell/d$ is divisible by $mn/d$, and hence by $m$.) Therefore, $\sum_{k=1}^{n-1} \sum_{i=1}^{n-1} \frac{im}{d} a_{ik} b_{kj}=0 \mod m$. 
Now by \eqref{eq:1}, we get that
$$
\sum_i \frac{im}{d}s_{ij}a_{ij}=0 \mod m.
$$
In particular, taking $j=1$, we get that $2m/d=0\mod m$. So, $d$ divides $2$. Hence, we must have $d=1$, since $d$ 
divides $m$ which is odd by \eqref{eq:order}. Therefore, such a twist $J_Q$ does not exist.

\smallskip

However, this issue can be remedied by considering the following Hopf algebra.

\begin{definition} The {\it small quantum group of adjoint type}, denoted by $\widetilde{u}_q({\mf{g}})$, is generated by $u_q({\mathfrak{g}})/(c-1)_{c \in C}$ and commuting grouplike elements 
$g_i$, subject to relations 
\begin{equation*}
g_ie_jg_i^{-1}=q^{\delta_{ij}} e_j,\quad g_if_jg_i^{-1}=q^{-\delta_{ij}} f_j, \quad g_i^m=1, \quad k_i=\prod_j g_j^{d_ia_{ij}}.
\end{equation*}
It has dimension $m^{\dim({\mf{g}})}$, and is related to the adjoint group of ${\mf{g}}$.
\end{definition}

Now $\widetilde{u}_q({\mf{g}})$ has a Hopf subalgebra $\widetilde{u}_q^{\ge 0}({\mf{g}})$, 
which acts inner faithfully on $A_{q,Q}$ (namely, the action is extended via $g_i\cdot z_j=q^{\delta_{ij}}z_j$). 
Let $G'$ be the group of grouplike elements of $\widetilde{u}_q^{\ge 0}({\mf{g}})$, and
let $\alpha_i$ be the generators of $\widehat G'$ defined by $\alpha_i(g_j)=q^{\delta_{ij}}$. 
Hence, the equation (\ref{sigmaj}) for $\sigma_{J_Q}$ has a unique solution and 
we obtain the result below.

\begin{proposition} \label{prop:uqgpostwist1} 
The twists $\widetilde{u}_q^{\geq 0}(\mf{g})^{J_Q}$  are Galois-theoretical. \qed
\end{proposition}

This provides $2^{\text{rank}({\mathfrak g})-1}$ Galois-theoretical Hopf algebras, without assuming the condition \eqref{eq:detaij} on $m$. 


\begin{remark} \label{rem:Milen} We thank Milen Yakimov for the following remark that, in fact, there is a different way to construct the $u_q^{\ge 0}({\mf{g}})$-module algebras $A_{q,Q}$. Namely, $A_{q,Q}$ arises as a coideal subalgebra of ${\mathcal{U}}_q^{\ge 0}({\mf{g}})$, and since $u_q^{\ge 0}({\mf{g}})$ is self-dual, $A_{q,Q}$ also arises as a $u_q^{\ge 0}({\mf{g}})$-module algebra. We see this as follows.

There are general classification results for coideal
subalgebras in ${\mathcal{U}}_q^{\ge 0}({\mf{g}})$ by Heckenberger-Schneider \cite{HS} and
by Heckenberger-Kolb \cite{HK}. The results are that under certain
natural conditions all (one-sided) coideal subalgebras are 
tensor products of the Cartan part of ${\mathcal{U}}_q^{\ge 0}({\mf{g}})$ with ${\mathcal{U}}^+[w]$ for $w \in W$ (the Weyl group).
The second factor is a $q$-analog of
$\mathcal{U}({\mathfrak{n}}_+ \cap w({\mathfrak{n}}_-)).$

All ${\mathcal{U}}^+[w]$ are iterated Ore extensions, which 
are $q$-polynomial rings if and only if $w$ has no repeating simple
reflections in one (hence, in every) reduced decomposition, that is to say, if and only if the  $w$ is a subexpression of a Coxeter element.
 Also, it is not hard to show that at roots of unity, the coaction
 of the quantum group on its coideal subalgebra descends to the small quantum group, and 
is inner faithful if and only if $w$ is a Coxeter element.

Therefore, the $\mathcal{U}^+[w]$ that (1) admit an inner faithful action of $u_q^{\geq 0}(\mf{g})$ and (2) are isomorphic to a $q$-polynomial algebras, are exactly 
those coming from the Coxeter elements of $W$. 

To relate this construction to our construction of an inner faithful $u_q^{\geq 0}(\mathfrak{g})$-module algebra, we need to define 
a bijection between orientations of the Dynkin diagram and Coxeter elements in $W$. 
Namely, an orientation of the Dynkin diagram defines a partial order on vertices,
and we can extend it to a total order and write the corresponding word $s_{i_1} \dots s_{i_r}$, which is a Coxeter element of $W$.
Then, one can show that any two such total orderings give the same
element of $W$. Conversely, given a Coxeter element, we can say that $i\rightarrow j$
if $s_i$ appears before $s_j$ in the word, and this defines an orientation on the Dynkin diagram. 
See \cite[Exercise~3.2]{GeckPfeiffer}.
\end{remark}


\subsection{Non-pointed Galois-theoretical Hopf algebras} \label{subsec:nonptedGT} Consider the following example from \cite{Etingof:PIintegrality}.

\begin{example} \label{ex:nonptedGT} \cite[Example~3.16]{Etingof:PIintegrality}
Let $n,m$ be positive integers and let $q$ be a primitive $m$-th root of unity. Consider the generalized Taft algebra $K = T(nm, m, 1)$ (from Section~\ref{subsec:genTaft}) generated by a grouplike element $g$ and a $(1,g)$-skew primitive element $x$, subject to relations $g^{nm}=1$, $x^m = g^m- 1$, and $gx=qxg$. We get that $K$ coacts inner-faithfully on $\kk(z)$ by the formula $\rho(z) = z \otimes g + 1 \otimes x$. We also have that $K$ is not basic. Thus, $H = K^*$ is a non-pointed Galois-theoretical Hopf algebra by Lemma~\ref{lem:Skryabin}.
\end{example}


\subsection{On duals and twistings of Galois-theoretical Hopf algebras}  \label{subsec:GTtwist}
We now discuss the preservation of the Galois-theoretical property under taking Hopf duals and twists.
The results about twists (parts (b) and (c) below) were observed by Cesar Galindo; we thank him for allowing us to use this result.

\begin{proposition} \label{prop:GTtwist}
The Galois-theoretical property is preserved neither under (a) Hopf dual, (b) 2-cocycle deformation (that alters multiplication), nor (c) Drinfeld twist (that alters comultiplication).
\end{proposition}

\begin{proof}
(a) Consider Example~\ref{ex:nonptedGT}: the Hopf dual of a generalized Taft algebra $T(nm,m,1)$ is Galois-theoretical. However, $T(nm,m,1)$ is not Galois-theoretical by Proposition~\ref{prop:genTaft}. More simply, one could also use a group algebra of a finite non-abelian group as a counterexample by Proposition~\ref{prop:GTprelim}(a,b).

(b) Consider Proposition~\ref{prop:uqsl2}: $u_q(\mf{sl}_2)$ is Galois-theoretical, yet its associated graded Hopf algebra gr($u_q(\mf{sl}_2)$) is not. Moreover, gr($u_q(\mf{sl}_2)$) is a 2-cocycle deformation of $u_q(\mf{sl}_2)$ by \cite[Theorem~7.8]{Masuoka}.

(c) Consider a Galois-theoretical group algebra $\kk G$ and take a nontrivial Drinfeld twist $J$ of $\kk G$ so that $(\kk G)^J$ is noncocommutative. Note that $(\kk G)^J$ is a semisimple Hopf algebra. So if $(\kk G)^J$ is Galois-theoretical, then by Proposition~\ref{prop:GTprelim}(b), $(\kk G)^J$ is a group algebra, which yields a contradiction.
\end{proof}

\section{Appendix}
This article has appeared in  
{\it Transformation Groups}, 20(4):985-1013, 2015. Here are the TG reference numbers versus the arXiv reference numbers: 
\medskip

\noindent Definitions 1-12 = 2.4, 2.6, 2.8, 2.9, 2.11, 2.12, 2.13, 2.14, 3.3, 4.12, 4.19, 4.31; 
\medskip

\noindent Theorems (T), Propositions (P), Lemmas (L), Corollaries (C), Conjectures (Cj): P1=1.1, T2=1.2, T3=1.3, L4=2.1, Cj5=2.2, T6=2.3, P7=2.15, T8=2.18, L9=3.1, P10=3.4, T11=3.6, L12=3.7, C13=3.8, T14=3.9, T15=3.10, L16=4.1, P17-27=4.2-4.11, 4.13, L28=4.14, P29=4.15, P30=4.16, C31=4.17, P32-39=4.20-4.24, 4.28, 4.32, 4.35;
\medskip

\noindent  Remarks 1-6=1.4, 2.10, 3.2, 4.18, 4.29, 4.33 ; Example 1= 4.34 ; Equations 1-6=2.5, 2.7, 4.25, 4.26, 4.27, 4.30; Questions 1=3.5.

\section*{Acknowledgments}
We thank the anonymous referees for making several suggestions that improved the quality of this manuscript, which include a shorter proof of Proposition~\ref{prop:UactonAqQ}. We are grateful to Susan Montgomery  for pointing out references \cite{MontgomerySchneider} and \cite{MontgomerySmith}. We thank Nicol\'{a}s Andruskiewitsch and Iv\'{a}n Angiono for
useful discussions and insightful suggestions. We also thank Cesar Galindo for his
contribution to Section~\ref{subsec:GTtwist}, and thank Milen Yakimov for supplying Remark~\ref{rem:Milen}. 
The authors were supported by the National Science Foundation: NSF-grants DMS-1000173, DMS-1102548, and DMS-1401207.

\bibliography{EW_NonssComDomain}

\end{document}